\documentclass[review]{elsarticle} 
\usepackage{lineno,hyperref} 
\modulolinenumbers[5] 

\journal{Journal of \LaTeX\ Templates} 

\usepackage[export]{adjustbox}
\usepackage{graphicx}  
\usepackage[caption=false]{subfig} 
\usepackage{graphicx}
\usepackage{amsthm}
\usepackage{amsmath} 
\usepackage{amssymb} 
\usepackage{mathtools}
\DeclarePairedDelimiter\ceil{\lceil}{\rceil}
\DeclarePairedDelimiterX{\inp}[2]{\langle}{\rangle}{#1, #2} 
\usepackage{acronym}
\usepackage[acronym]{glossaries}
\makeglossaries
\newacronym{TV}{TV}{\emph{Total Variation}}
\newacronym{FFT}{FFT}{\emph{Fast Fourier Transform}}
\newacronym{R-L}{R-L}{\emph{Reimann-Liouville}}
\newacronym{FTFC}{FTFC}{\emph{Fundamental Theorem of Fractional Calculus}}
\newacronym{PET}{PET}{\emph{Positron Emission Tomography}}
\newacronym{PDE}{PDE}{\emph{Partial Differential Equation}}

\makeglossaries
\usepackage[capitalise]{cleveref}
\newtheorem{assumption}{Assumption}
\crefname{assumption}{Assump.}{Assumption}
\newtheorem{definition}{Definition}
\crefname{definition}{Def.}{Definition}
\newtheorem{theorem}{Theorem}
\crefname{theorem}{Theorem}{Theorems}
\newtheorem{remark}{Remark}
\crefname{remark}{Remark}{Remarks}

\crefname{lemma}{Lemma}{Lemmas}
\newtheorem{proposition}{Proposition}
\crefname{proposition}{Prop.}{Propositions}
\newtheorem{corollary}{Corollary}
\crefname{corollary}{Cor.}{Corollaries}
\crefname{section}{Sec.}{Sections}
\crefname{algorithm}{Algo.}{Algorithms}
\usepackage{physics} 
\usepackage{xcolor}

\begin{document}
\begin{frontmatter}
\title{Introducing the $p$-Laplacian Spectra}
\author{Ido Cohen\footnote{Corresponding author. E-mail addresses: idoc@campus.technion.ac.il,ido.coh@gmail.com}}
\author{Guy Gilboa}
\address{Department of Electrical Engineering, Technion - Israel Institute of Technology, 3200003
Haifa, Israel}




\begin{abstract}
In this work we develop a nonlinear decomposition,  associated with nonlinear eigenfunctions of the $p$-Laplacian for $p\in(1,2)$. With this decomposition we can process signals of  different degrees of smoothness.

We first analyze solutions of scale spaces, generated by $\gamma$-homogeneous operators, $\gamma\in\mathbb{R}$. An analytic solution is formulated when the scale space is initialized with a nonlinear eigenfunction of the respective operator. We show that the flow is extinct in finite time for $\gamma\in[0,1)$.

A main innovation in this study is concerned with operators of fractional homogeneity, which require the mathematical framework of fractional calculus. The proposed transform rigorously defines the notions of decomposition, reconstruction, filtering and spectrum. The theory is applied to the $p$-Laplacian operator, where the tools developed in this framework are demonstrate\textcolor{black}{d}.
\end{abstract}

\begin{keyword}
Nonlinear spectra, filtering, shape preserving flows, p-Laplacian, nonlinear eigenfunctions.
\end{keyword}

\end{frontmatter}

\printglossary[type=\acronymtype]

\section{Introduction}
Data representation is commonly performed by transforming a signal into a different domain, more convenient for analysis and processing. By the Fourier transform, a signal can be represented as a sum of eigenfunctions of the Laplace operator \cite{fourier1822theorie}. Recently, a new data representation with respect to the $1$-Laplacian operator was suggested in \cite{gilboa2013spectral,gilboa2014total}, referred to as the \acrshort{TV}-transform. These transforms are based on operators which are particular cases of the $p$-Laplace operator,
\begin{equation}\label{eq:pLaplaceOperator}
    \Delta_p\left(u\right)=\textrm{div}\left(\abs{\nabla u}^{p-2}\nabla u\right),\quad p\in[1,\infty),
\end{equation}
where $u$ resides in some Hilbert space and $\nabla$ is the gradient operator. The $p$-Laplace operator is the negative variational derivative of the $p$-Dirichlet energy,
\begin{equation}\label{eq:pEnergy}
    J_p(u)=\frac{1}{p}\inp{\abs{\nabla u}^p}{1}.
\end{equation}
The Fourier and \acrshort{TV}-transforms stem from the gradient descent of the $p$-Dirichlet energy, $u_t(t) = \Delta_p u(t)$, with some initial condition. This flow becomes the heat equation for $p=2$ or the \acrshort{TV}-flow for $p=1$ \cite{andreu2001minimizing}. The data representation, achieved by means of these transforms, are related to the (non)-linear eigenfunctions, defined by
\begin{equation}\label{eq:pLaplaceEF}
    \Delta_p(\phi)=\lambda \cdot \phi,\quad \lambda\in\mathbb{R},
\end{equation}
with the respective value of $p$. Eq. \eqref{eq:pLaplaceEF} should be complemented with suitable boundary conditions, yielding different eigenfunctions. In this paper, we assume Neumann boundary conditions.
The existence of eigenfunctions in the form of \eqref{eq:pLaplaceEF} was shown in \cite{garcia1987existence} for $p\in[1,\infty)$. Analytic solutions of \eqref{eq:pLaplaceEF} are known only for $p=1,2$, as far as we know. However, for $1<p<2$ the eigenfunctions can be numerically computed, for instance using \cite{cohen2018energy}. Eigenfunctions for different value of $p$ with Neumann boundary conditions are shown in Fig. \ref{Fig:eigenFun1dVaryP2-1}. Note, the eigenfunctions change in a continuous manner from a cosine function ($p=2$) to a step function (p=1). There are no general analytic solutions for the nonlinear $p$-Laplacian flow. We remark that the weak solution of $u_t(t)=\Delta_p(u(t))$, when it is initialized with a Dirac measure, is known as the Barenblatt solution \cite{barenblatt1952self} as explained in \cite{kamin1988fundamental}. For example, for $p=2,\,x\in\mathbb{R}^N$ we obtain a Gaussian kernel with variance proportional to $t$.

The $p$-Laplacian scale space, for $p\in[1,2]$, has increasingly attracted attention in recent years \cite{liu2016renormalized,chen2006variable,baravdish2015backward,huang2016level,chen2010image,yi2017variable,maiseli2015multi,baravdish2018damped,wei2012p}. In the studies \cite{kuijper2007p,kuijper2013image,kuijper2009geometrical} by Kuijper, this flow was analyzed in terms of gauge coordinates. Blomgren et al. \cite{blomgren1997total} suggested to design a flow with adaptive $p$, as $p$ directly controls edge preservation.

Other than image denoising and analysis, the $p$-Laplacian operator plays a crucial role in semi-supervised learning. A generalizing function in this domain is essential when there is significantly fewer labeled data compered to unlabeled \cite{calder2018game,flores2018algorithms}. See recent advances on this topic in  \cite{liu2018p,ma2018hypergraph}.
%
%
\begin{figure}
\centering
\includegraphics[width=.625\textwidth]{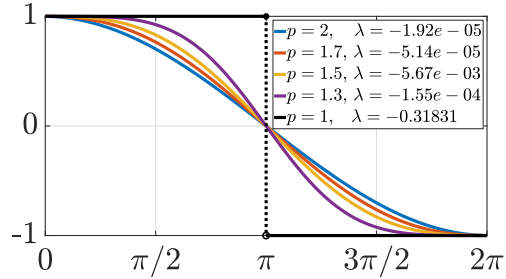}
\caption{Examples of eigenfunctions with increasing homogeneity, $p$. The edges are emphasized when $p\to 1$.}
\label{Fig:eigenFun1dVaryP2-1}
\end{figure}

Thus, signal representation based on $p$-Laplacian eigenfunctions for $1<p<2$ is called for. The suggested nonlinear decomposition, referred to as the \emph{$p$-transform}, is a generalization of the transforms in \cite{gilboa2013spectral,gilboa2014total,burger2016spectral}.
With the $p$-transform we can interpolate between different degrees of smoothness and unify the transform formulation in one expression for $p\in(1,2)$.

The main contributions  of this work are:
\begin{enumerate}
    \item Analytic solutions for evolutions based on $\gamma$-homogeneous operators are formulated for a certain class of initial conditions (nonlinear eigenfunctions).
    \item A signal decomposition related to the eigenfunctions of a $\gamma$-homogeneous operator, where $\gamma\in(0,1)$, is formulated.
    \item This facilitates signal decompositions with different degrees of smoothness.
    \item The proposed framework rigorously defines novel nonlinear decomposition, reconstruction, spectrum and filtering.
\end{enumerate}

The plan of this paper is as follows. We begin by examining general scale space flows. We put a special emphasize on flows, generated by homogeneous operators (e.g. $p$-Laplace operator).  We formulate necessary and sufficient conditions for obtaining a solution with separated variables. This allows the formulation of an analytic solution of these flows,
initialized with an eigenfunction (Sec. \ref{sec:ShP_FET}). Based on the analytic solution and the finite extinction time of this solution, we propose a transform that represents an eigenfunction as a Dirac delta in the transform domain (Sec. \ref{sec:pspectra}). A rigorous signal analysis framework is formulated, including spectrum, filtering, and reconstruction. We demonstrate this framework, using the $p$-Laplacian operator for different values of $p$ (Sec. \ref{sec:experiments}). Conclusions are discussed in Sec. \ref{sec:conc}.
The detailed proofs are provided in Appendices (in order to improve the reading flow). We precede with preliminary definitions and identities.

\section{Preliminaries}\label{sec:MatMot}
We recall some essential definitions that are used in the paper. Let $\mathcal{H}$ be an infinite dimensional Hilbert space, which is equipped with an Euclidean inner product, $\inp{\cdot}{\cdot}$, and norm $\norm{\cdot}=\sqrt{\inp{\cdot}{\cdot}}$. 

\begin{definition}[Monotone operator] Let $P$ be an operator on $\mathcal{X}$ in some Hilbert space $H$. Then, $P$ is a monotone operator if the following holds
\begin{equation}\label{eq:monotonicOperator}
    \inp{P(u)-P(v)}{u-v} \ge 0, \quad \forall u,v \in \mathcal{X}.
\end{equation}
\end{definition}
\begin{definition}[Maximally monotone operator]
A monotone operator with setting domain $\mathcal{X}$ is maximally monotone if there is no other monotone operator with a setting domain which properly contains $\mathcal{X}$.
\end{definition}
\begin{definition}[Kernel]The kernel of the operator $P$ is a set, $\mathcal{K}\subset \mathcal{X}$, defined by
\begin{equation}\label{eq:Kernel}
    \mathcal{K} = \left\{e\in\mathcal{X}:P(e)=0\right\}.
\end{equation}
\end{definition}
\begin{definition}[Orthogonal complement of the kernel]The orthogonal complement of the kernel, $\mathcal{K}$, of the operator $P$, is defined by
\begin{equation}\label{eq:OrthoKernel}
    \mathcal{K}^\perp = \left\{w\in\mathcal{X}:\inp{w}{e}=0,\quad \forall e\in\mathcal{K} \right\}.
\end{equation}
\end{definition}
\begin{definition}[Kernel orthogonality]\label{def:KernelOrthogonality}
We say $P$ is kernel orthogonal if $Img\{P\}\subseteq \mathcal{K}^\perp$, i.e.
\begin{equation}
    \inp{P(u)}{e}=0,\quad \forall e\in \mathcal{K}, \, \forall u\in \mathcal{X}.
\end{equation}
\end{definition}

\begin{definition}[Nonlinear flow]
We define a nonlinear flow by
\begin{equation}\label{eq:ss}\tag{\bf{Flow}}
\begin{split}
    u_t(t) = P(u(t)),  \qquad u(0)=f\in\mathcal{X},
\end{split}
\end{equation}
where $-P(\cdot)$ is a maximally monotone operator,  $P(0)=0$, and $f$ is the initial condition. The solution of \eqref{eq:ss} exists according to \cite{brezis1973ope}. We assume the solution is unique and differentiable, with respect to $t\in (0,\tau),\, \tau\in\mathbb{R}_{>0}$.
\end{definition}
\begin{definition}[Mass preserving flow]\label{def:GeneralizedMassConservation}
The nonlinear scale space \eqref{eq:ss} is mass preserving if 
\begin{equation*}
    \inp{u(t)}{e}=const,\quad\forall e\in \mathcal{K},\forall t\in \mathbb{R}.
\end{equation*}
\end{definition}
Mass conservation can be achieved if $Img\{P\}\perp Ker\{P\}$, i.e. if $P$ meets Definition \ref{def:KernelOrthogonality}.

\begin{definition}[Coercive operator]\label{def:CoerciveOperator}
Let $P$ be a $\gamma$-homogeneous operator. We say $P$ is a coercive operator if there is a positive constant $C$ that
\begin{equation}\label{eq:CoerciveOperator}
    \norm{u}^{\gamma +1}\leq C\inp{-P(u)}{u}, \quad \forall u\in K^\perp.
\end{equation}
\end{definition}

\begin{definition}[Homogeneity]
We say that $P$ is a $\gamma$-homogeneous operator (not necessarily integer) if the following relation holds
\begin{equation}\label{eq:homoOpe}
    P(cu)=c\abs{c}^{\gamma-1}\cdot P(u),\quad \forall c\in \mathbb{R}.
\end{equation}
\end{definition}
In the context of this paper, we formulate the gradient descent of \eqref{eq:pEnergy} as follows.
\begin{definition}[$p$-Flow] The $p$-flow is the gradient flow with respect to the $p$-Dirichlet energy, \eqref{eq:pEnergy},
\begin{equation}\label{eq:pFlow}\tag{\bf pFlow}
    \begin{split}
        u_t(t) =& \Delta_p u(t),\quad u(0)=f\in\mathcal{K}^\perp,
    \end{split}
\end{equation}
with homogeneous Neumann boundary conditions, where $\Delta_p$ is the $p$-Laplacian operator, Eq. \eqref{eq:pLaplaceOperator}.
\end{definition}
In this paper we focus on nonlinear flows where $P$ is a homogeneous operator. Note, that the $p$-Laplacian is a $p-1$ homogeneous operator. The $p$-Laplacian operator is maximally monotone since it is a subgradient of a convex functional. We will often omit the spatial index, $x$, or the temporal one, $t$, when the context is clear.
\begin{definition}[Finite extinction time] Let $u(t)$ be a solution of \eqref{eq:ss}. If there exists $T\in\mathbb{R}_{>0}$ such that
\begin{equation}\label{eq:defExtinctionTime}
\begin{split}
    u_t(t)=0, \quad \forall t\ge T
\end{split}
\end{equation}
then we say that \eqref{eq:ss} has a finite extinction time $T$, where $T$ is the minimal value of $t$ for which \eqref{eq:defExtinctionTime} holds.
\end{definition}
Note, that if the flow is mass preserving and the initial condition $f$ belongs to $\mathcal{K}^\perp$ then \eqref{eq:defExtinctionTime} implies $u(t)=0$ for all $t>T$.

\begin{definition}[Shape preserving flow]
We term a solution, $u(t)$, of \eqref{eq:ss} as a shape preserving flow if it has separated variables, i.e. 
\begin{equation}\label{eq:shapePreserving}\tag{\bf{SPF}}
\begin{split}
    u(t) = a(t)\cdot f.
\end{split}
\end{equation}
\end{definition}
The spatial shape of $u(t)$, $f$, is preserved over time whereas only the contrast changes. We term the contrast, $a(t)$, as the \emph{decay profile}. 
\begin{definition}[Nonlinear eigenfunction] $v\in \mathcal{X}$ is a nonlinear eigenfunction of the (nonlinear) operator $P$ if there exists $\lambda\in\mathbb{R}$ such that
\begin{equation}\label{eq:ef}\tag{\bf EF}
P(v) = \lambda \cdot v.
\end{equation}
Since $-P$ is assumed to be a maximally monotone operator the eigenvalue is non-positive. This can be easily verified by $\lambda \norm{v}^2 = \inp{P(v)}{v}=\inp{P(v)-P(0)}{v-0}\leq 0$ for any eigenfunction $v$.
\end{definition}
We will often refer to a nonlinear eigenfunction simply as an eigenfunction.
\subsection*{Fractional calculus background}
There are several definitions of fractional derivatives and integrals \cite{dalir2010applications}. We write below the most established ones, which are used in this paper.  
\begin{definition}[Fractional integrals in the Riemman-Liouville sense \cite{grigoletto2013fractional}]
Let $y$ be a real function $y:[a,b]\to \mathbb{R}$ in $L^1[a,b]$. The left sided and the right sided fractional integrals of order $\alpha\in\mathbb{R}_{> 0}$ are respectively defined as
\begin{subequations}\label{eq:LiouvilleIntegral}
\begin{equation}\label{eq:lowerLiouvilleIntegral}
    {I_{a^+}^\alpha\left\{y\right\}(t)}= \frac{1}{\Gamma(\alpha)}\int_{a}^t(t-\tau)^{\alpha -1}y(\tau)d\tau   
\end{equation}
\begin{equation}\label{eq:upperLiouvilleIntegral}
    I_{b^-}^\alpha\left\{y\right\}(t)= \frac{1}{\Gamma(\alpha)}\int_t^b (\tau-t)^{\alpha -1}y(\tau)d\tau,
\end{equation}
\end{subequations}
where $\Gamma(\cdot)$ is the extension of the factorial function to the real axis, for a positive integer $n$ we have $\Gamma(n)=(n-1)!$.  
\end{definition}
Explanations on the origins of the left and right handed definitions are given in \cite{richard2014fractional} page 36. The left sided definition is valid for $t>a$ and the right sided for $t<b$. In principle, these are direct extensions of Cauchy's formula. In this work we mostly use the right sided definition \eqref{eq:upperLiouvilleIntegral}. 
\begin{definition}[Fractional derivatives in the Riemman-Liouville sense \cite{grigoletto2013fractional}]
Let $y$ be a real function $y:[a,b]\to \mathbb{R}$ in $L^1[a,b]$. The right handed fractional derivative of order $\alpha\in \mathbb{R}_{\ge 0}$ of $y$ is
    \begin{equation}\label{eq:LiouvilleDerivative-}
        D_{b^-}^\alpha\left\{y\right\}(t)=(-1)^{\ceil{\alpha}} \frac{d^{\ceil{\alpha}}}{dt^{\ceil{\alpha}}}\left\{I_{b^-}^{{\ceil{\alpha}}-\alpha}\left\{{y(t)}\right\}\right\},
    \end{equation}
where $\ceil{\alpha}$ is the least integer greater than $\alpha$, and
$I_{b^-} \left\{{y(t)}\right\}$ is defined in \eqref{eq:upperLiouvilleIntegral}. 
\end{definition}

\begin{definition}[Space $I^\alpha_{b^-}\left(L^p\right)$ \cite{grigoletto2013fractional}]
The space $I^\alpha_{b^-}\left(L^p\right)$ is defined, for any $\alpha\in\mathbb{R}_{>0}$ and $p\ge1$, by
\begin{equation*}
    I^\alpha_{b^-}\left(L^p\right) :=\left\{y(t):y(t)=I^\alpha_{b^-}\left\{x(t)\right\},\, x(t)\in L^p(a,b)\right\}.
\end{equation*}
\end{definition}
The theoretical basis of the \acrshort{FTFC} \cite{grigoletto2013fractional} 
can be found in \cite{kilbas2006theory} Lemmas 2.4-2.7 and in \cite{samko1993fractional} Theorem 2.4.
\begin{definition}[\acrlong{FTFC}]
If $y(t)\in I^\alpha_{b^-}(L^1)$ then 
\begin{equation}\label{eq:FTFC}
    I^\alpha_{b^-}\left\{D^\alpha_{b^-}\left\{y\right\}\right\}(t)=y(t).
\end{equation}
\end{definition}
The aforementioned definitions hold for $\alpha\in\mathbb{C}$ with a positive real part. However, in our work $\alpha$ is limited to the real numbers.

\begin{assumption}\label{ass:flowAssumption}In this paper we analyze the scale space, Eq. \eqref{eq:ss}, where $P$ is a $\gamma$-homogeneous operator and admits the Definitions \ref{def:GeneralizedMassConservation} and \ref{def:CoerciveOperator}.
\end{assumption}
If $P$ is a subgradient of a convex homogeneous functional part of the Assumption \ref{ass:flowAssumption} can be proven \cite{bungert2019asymptotic}.
\textcolor{black}{
We note that subgradients admit mass preservation. Coercivity of the operator implies coercivity of the functional. For functionals which are norms, coercivity is given, in finite dimensions, due to norm equivalence.}

\section{Towards the $p$-transform - shape preserving flows}\label{sec:ShP_FET}
The \acrshort{TV}-transform \cite{gilboa2013spectral,gilboa2014total} is motivated by the analytic solution of \eqref{eq:pFlow}, for $p=1$, initialized with a nonlinear eigenfunction \eqref{eq:ef}.
The analytic solution is
\begin{equation}\label{eq:TVsolution}
    u(x,t)=\left(1+\lambda\cdot t\right)^+\cdot f(x),
\end{equation}
where $q^+=\max\{0,q\}$ \cite{bellettini2002total}. We refer to three attributes of this solution as the theoretical basis of the \acrshort{TV}-transform: the solution is in the form of \eqref{eq:shapePreserving}, its decay profile is linear, and it extincts in finite time. 

These attributes are studied here for a homogeneous operator, which were firstly introduced in \cite{cohen:hal-01870019}. In general, the necessary condition of the solution to be in the form of \eqref{eq:shapePreserving} is that the initial condition $f$ is an eigenfunction. The following theorem asserts that if $P$ is a homogeneous operator this condition is also sufficient.

\subsection*{Shape preserving flows for homogeneous operators}
\begin{theorem}\label{theo:homoOperator} Let the solution of \eqref{eq:ss} exist and $P$ be a $\gamma$-homogeneous operator. The solution is shape preserving iff the initial condition $u(t=0)=f$ is an eigenfunction with eigenvalue $\lambda$, i.e. admits \eqref{eq:ef}. In that case, the decay profile is 
\begin{equation}\label{eq:u_analytic}
a(t)=\left[\left((1-\gamma)\lambda t+1\right)^+\right]^{\frac{1}{1-\gamma}}.
\end{equation}
\end{theorem}
The corresponding proof can be found in \ref{sec:appProofTheohomoOperator}. Applying this theorem to the $p$-Laplacian operator, we can conclude the following.
\begin{corollary} If \eqref{eq:pFlow} is initiated with an eigenfunction, $f$, then
\begin{equation}\label{eq:pFlowAnalytic}
    u(t) = a(t)\cdot f = \left[\left((2-p)\lambda t+1\right)^+\right]^{\frac{1}{2-p}}\cdot f.    
\end{equation}
\end{corollary}
It is simply shown by setting $\gamma=p-1$ in Theorem \ref{theo:homoOperator}. In addition, for $1\le p<2$ and $\lambda<0$ the $p$-flow has finite extinction time $T$
\begin{equation}\label{eq:gExtinctionTime}
    T=\frac{1}{(p-2)\lambda}.
\end{equation}
For $p>2$ the $p$-flow does not vanish. For 
$p\to2$ the solution \eqref{eq:pFlowAnalytic} becomes
\begin{equation*}
u(t) = e^{\lambda t}f.
\end{equation*}
This can be shown by using the limit $\left(1+\frac{1}{n}\right)^n \underset{n\to \infty}{\longrightarrow}e$. This result is expected from the heat equation. Eq. \eqref{eq:pFlowAnalytic} coincides with \eqref{eq:TVsolution} for $p=1$. 

\paragraph{Extinction time for arbitrary initial conditions} 

\begin{proposition}[Finite extinction for a coercive $\gamma$-homogeneous operator]\label{prop:finiteTimeExtinction} Let $P$ a $\gamma$-homogeneous and coercive operator in the sense of \eqref{eq:CoerciveOperator}. Let $u(t)$ be the solution of mass preserving flow \eqref{eq:ss} where $f\in \mathcal{K}^\perp$. Then there \textcolor{black}{exists $C>0$ such that $u(t)=0\,\,\forall t\geq T=\frac{\norm{f}^{1-\gamma}}{(1-\gamma)C}>0$.}
\end{proposition}
The corresponding proof can be found in \textcolor{black}{\cite{bungert2019asymptotic} (Theorem 2.13)}.

\section{The $p$-Framework}\label{sec:pspectra}
We now formulate a framework, which includes decomposition (transform), reconstruction (inverse-transform), spectrum, and filtering. In this section, we limit ourselves to a finite dimensional Hilbert space, i.e. $f\in \mathcal{X}\subset \mathcal{H}$ where $\mathcal{H}$ is a finite dimensional Hilbert space. Moreover, we limit ourselves to the flow and the operator $P$ that admit the conditions of Proposition \ref{prop:finiteTimeExtinction}. Note that the scale (time) parameter is continuous.
We first define the fundamentals of this transform, then we discuss
its attributes.

\subsection{Definitions}
Let us denote
\begin{equation}\label{eq:orderOfDecay}
    \beta=\frac{1}{1-\gamma}.
\end{equation}
\begin{definition}[$p$-Transform]\label{def:p-LapSpec}
The $p$-transform is defined by,
\begin{equation}\label{eq:pTransform}
\begin{split}
    \phi(x,t) = \frac{t^{\beta}}{\Gamma(\beta+1)}{D}^{\beta + 1}_{b^-}\{u(x,t)\},
\end{split}
\end{equation}
where $u(t)$ is the solution of \eqref{eq:ss}, generated by $\gamma$-homogeneous operator, where $\gamma\in(0,1)$, $\beta$ is defined in \eqref{eq:orderOfDecay} and $D^{\beta+1}_{b^-}$ is defined in \eqref{eq:LiouvilleDerivative-} and is with respect to $t$. In addition, $b$ is greater than the extinction time of $u(t)$.
\end{definition}
\begin{definition}[Inverse $p$-transform]\label{def:invP-LapSpec} The inverse $p$-transform is defined by
\begin{equation}\label{eq:pInverseTransform}
\begin{split}
    \hat{f}(x) = \int_0^\infty\phi(x,t)dt.
\end{split}
\end{equation}
\end{definition}
\begin{definition}[Filtering]
Let $h(t)$ be a real function (a filter). The filtering of $f$ by the filter $h(t)$ is 
\begin{equation}\label{eq:recovFilter}
\begin{split}
    f_h(x)=\int_0^\infty \phi(x,t)\cdot h(t)dt.
\end{split}
\end{equation}
\end{definition}
The filtering in the transform domain, $t$, is a simple amplification (or attenuation) of $\phi(x,t)$ at every scale $t$. 

\begin{definition}[$p$-Spectrum]\label{def:pSpectrum}
The spectrum of $f$ at any scale $t$ is defined by,
\begin{equation}\label{eq:pSpectrum}
S(t)=\inp{f}{\phi(t)}.
\end{equation}
\end{definition}

\subsection{Attributes}
\subsubsection{The ``frequency'' representation of an eigenfunction}
\begin{theorem}[The $p$-transform of an eigenfunction]\label{theo:EFspectrum}
The $p$-transform, \eqref{eq:pTransform}, of an eigenfunction $f$ with eigenvalue $\lambda$ is:
\begin{equation}\label{eq:ptransformEF}
\begin{split}
    \phi(x,t) = f(x)\cdot t^{\beta}\left[(\gamma -1)\lambda\right]^{\beta+1}\cdot\delta(1-(\gamma -1)\lambda\cdot t),
\end{split}
\end{equation}
where $\delta(\cdot)$ is the Dirac delta function.
\end{theorem}
The proof can be found in \ref{sec:appProofTheoEFspectrum}.
\subsubsection{Reconstruction}
\begin{proposition}[Reconstruction]\label{pro:invP-LapSpec} Let $u(t)$ be the solution of \eqref{eq:ss} and belong to $I^{\beta + 1}_{b^-}\left(L^1\right)$. Then $f=u(0)$ can be reconstructed by the inverse transform, Def. \ref{def:invP-LapSpec}, $f=\hat{f}$. 
\end{proposition}
\begin{proof} Let us examine expression \eqref{eq:pInverseTransform}. From the finite extinction, there exists $T>0$ which admits \eqref{eq:defExtinctionTime}. Moreover, the solution is zero from that time as $f\in\mathcal{K}^\perp$. Then, $\phi(t)$ has a finite support in time. Reminding $b>T$, we can rewrite \eqref{eq:pInverseTransform} as  
\begin{equation*}
    \begin{split}
       \hat{f}=\int_0^\infty\phi(\tau)d\tau &= \int_0^{b^-}\frac{\tau^{\beta}}{\Gamma(\beta+1)}{D}^{\beta + 1}_{b^-}\{u\}(\tau)d\tau\\
        &\underbrace{=}_{\textrm{Eq. }\eqref{eq:upperLiouvilleIntegral}}\eval{I^{\beta+1}_{b^-}\left\{D^{\beta+1}_{b^-}\left\{u\right\}\right\}(t)}_{t=0}
        \underbrace{=}_{\textrm{Eq. } \eqref{eq:FTFC}}u(0)=f.
    \end{split}
\end{equation*}
\end{proof}
\begin{remark}
The assumption on the solution, $u$, might be relaxed by assuming that $u(t)\in L^1(0,b)$ and using the relevant formulation of the \acrlong{FTFC} (Theorem 2 in \cite{grigoletto2013fractional}).
\end{remark} 

\subsubsection{Filtering}
As shown in \eqref{eq:gExtinctionTime}, the extinction time is inverse proportional to the absolute eigenvalue for all $\gamma\in[0,1)$. It is well known that in linear decomposition (such as the Fourier transform) every eigenfunction is represented as a delta function in the transform domain. 
We follow the definitions of ideal filters in \acrshort{TV}-spectral decomposition (see \cite{gilboa2014total}) and formulate an ideal low pass filter with a cutoff at $t_1$ by
\begin{equation}\label{eq:idealLPF}
h_{LPF,t_1}(t) =
    \begin{cases}
        0& t<t_1\\
        1& t\geq t_1
     \end{cases}.
\end{equation}
The ideal high pass, band pass, and band stop filters can be defined in a similar manner (see \cite{gilboa2014total}). Another special case of low pass filter (not ideal) is
\begin{equation}\label{eq:liouvilleFilter}
h(t) =
    \begin{cases}
        0& t<t_1\\
        \left[\frac{t-t_1}{t}\right]^\beta& t\geq t_1
     \end{cases}.
\end{equation}
Substituting \eqref{eq:liouvilleFilter} and \eqref{eq:pTransform} in \eqref{eq:recovFilter} yields the \acrshort{FTFC}, Eq. \eqref{eq:FTFC}, at $t=t_1$, which immediately gives us
\begin{equation}
    f_{h}(x) = \int_0^\infty \phi(x,t) h(t)dt= u(x,t_1),
\end{equation}
where $u(x,t_1)$ is the solution of \eqref{eq:ss} at time $t_1$. Thus, the scale space can be interpreted as a specific type of LPF.

\subsubsection{Parseval-type identity}
Based on Def. \ref{def:pSpectrum}, the following Parseval-type identity holds,
\begin{equation*}
\begin{split}
\int_0^\infty S(t) dt=&\int_0^\infty \inp{f}{\phi(t)} dt
=\inp{f}{\int_0^\infty \phi(t)dt}
=\inp{f}{f}=\norm{f}^2.
\end{split}
\end{equation*}
\paragraph{Relation to the \acrshort{TV}-transform} The methodology we use to formulate the $p$-transform is inspired by the \acrshort{TV}-transform \cite{gilboa2013spectral,gilboa2014total} and  by its extension to one-homogeneous functional  decomposition 
\cite{burger2016spectral,bungert2019nonlinear,schmidt2018inverse}. It was found useful in several image-processing applications, e.g. for denoising \cite{moeller2015learning}, segmentation \cite{zeune2017multiscale} and image fusion \cite{hait2018spectral}. \textcolor{black}{The regularity properties of the time derivative for spectral-TV are addressed in \cite{bungert2019nonlinear} }. It can be seen that the $p$-transform is a generalization of the previous studies and the \acrshort{TV}-transform is obtained by assigning $\gamma=0$ ($\beta=1$). 

\begin{figure}[htb]
\captionsetup[subfigure]{justification=centering}
\subfloat[Eigenfunction (1D), $\lambda=-0.0059$]
  {
\includegraphics[width=0.5\textwidth]{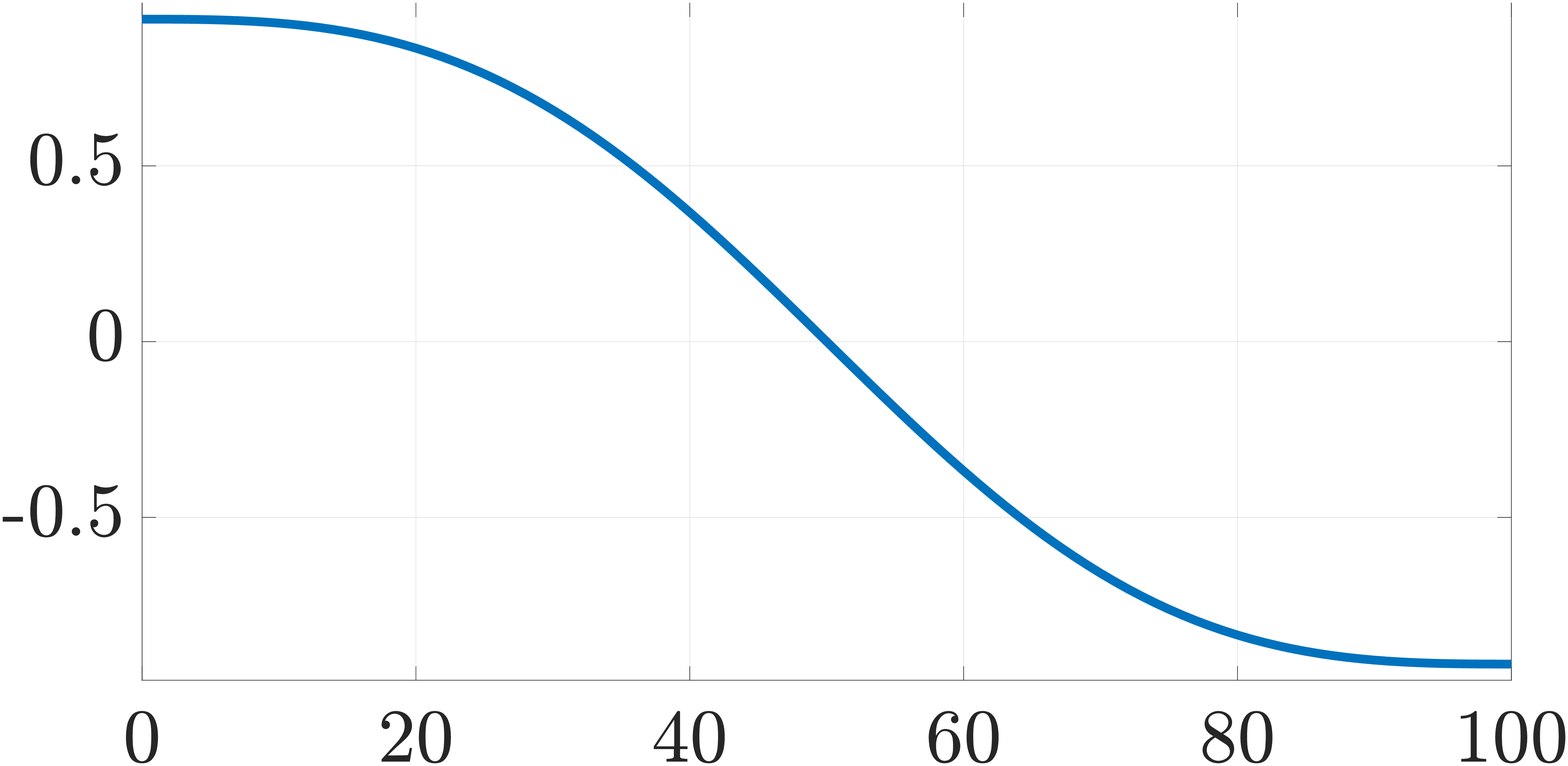}
}
\subfloat[{The 1D eigenfunction decay (theoretical and numerical results) vs. time}]
  {
\includegraphics[width=0.5\textwidth]{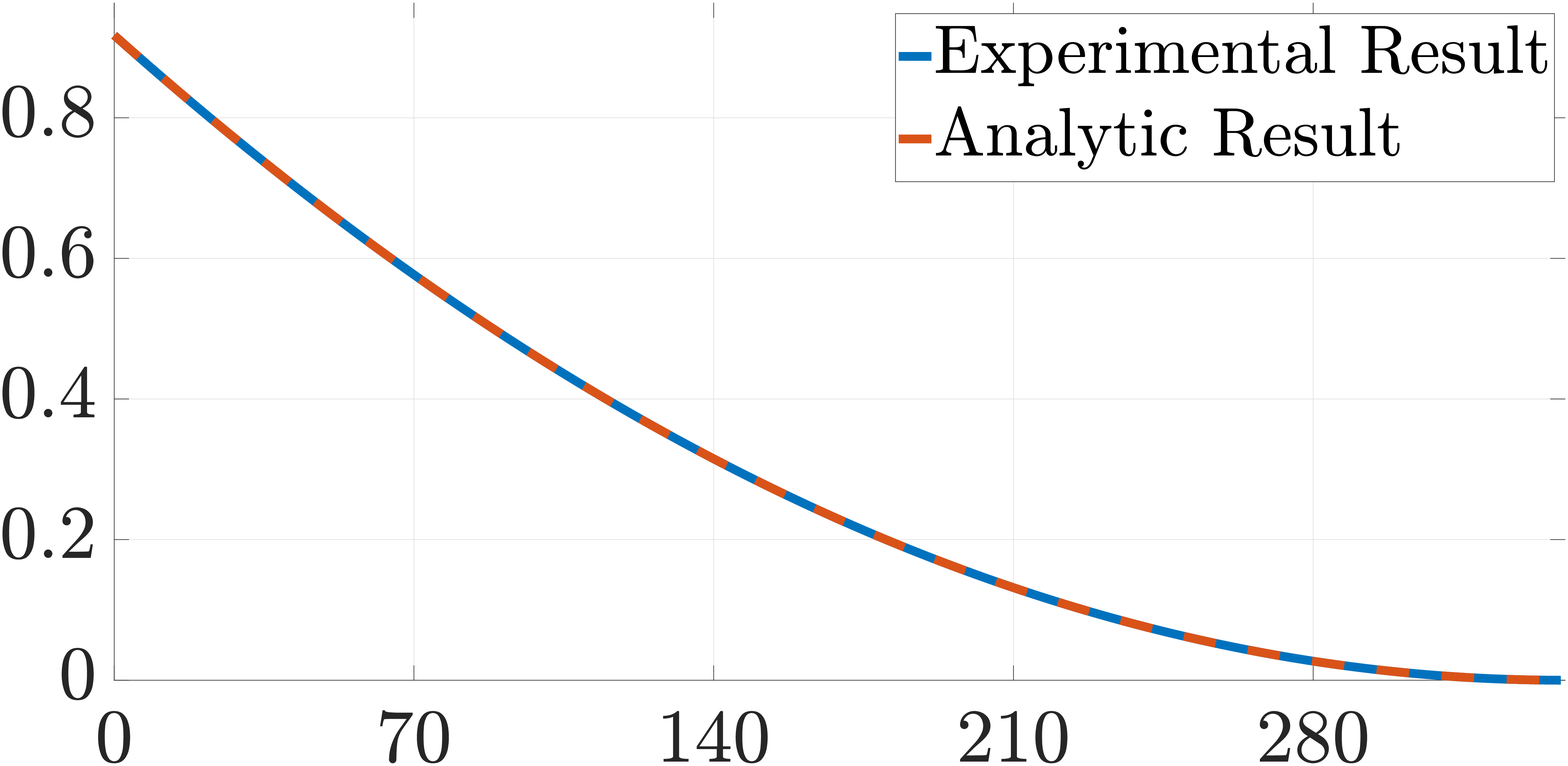}
\label{subfig:decay1D}
}\\
\subfloat[Eigenfunction (2D), $\lambda=-0.0269$]
  {\includegraphics[width=0.5\textwidth]{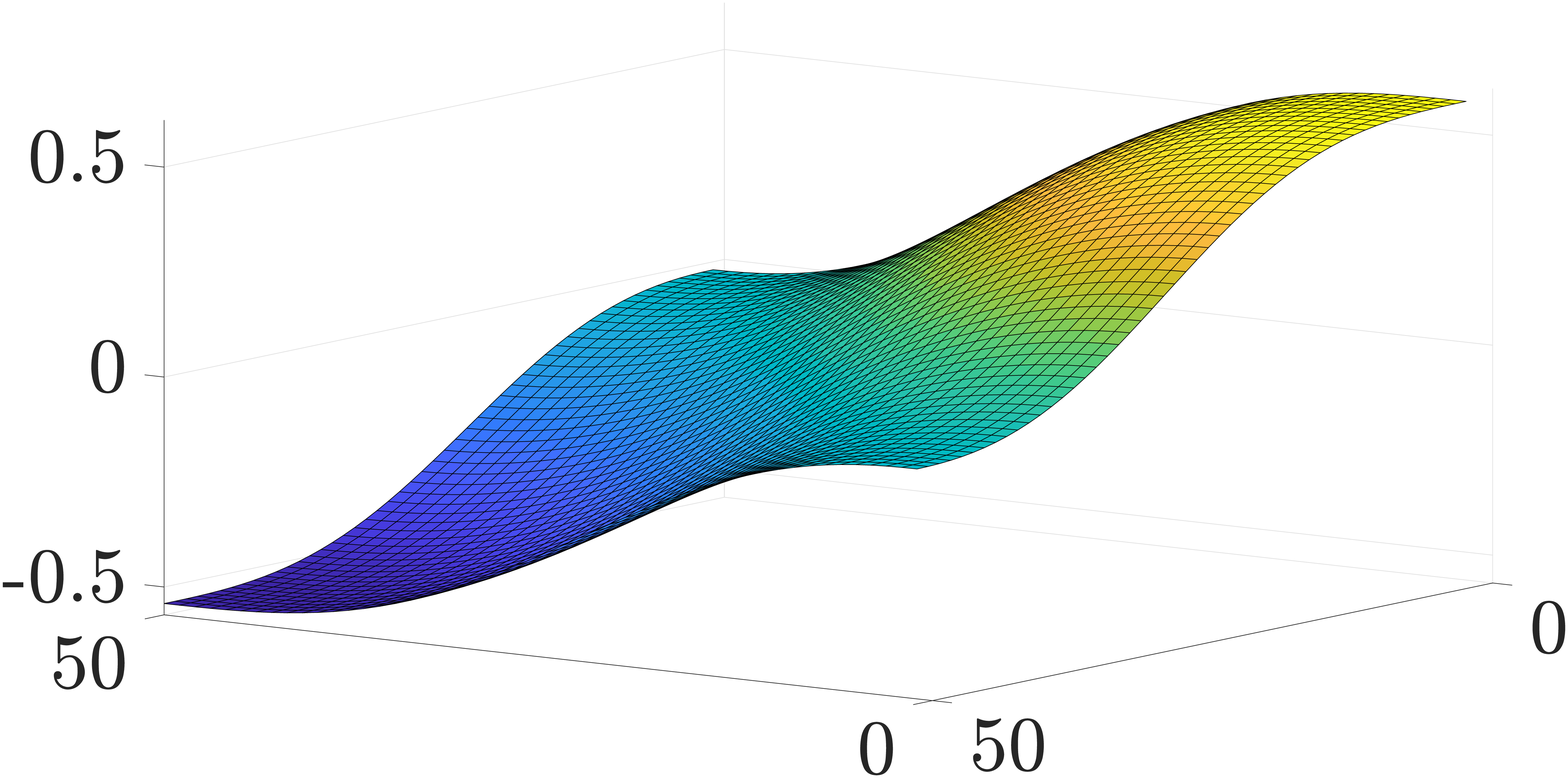}
\label{subfig:eigenFunc2D}
}
\subfloat[{The 2D eigenfunction decay (theoretical and numerical results) vs. time}]
 {
\includegraphics[width=0.5\textwidth]{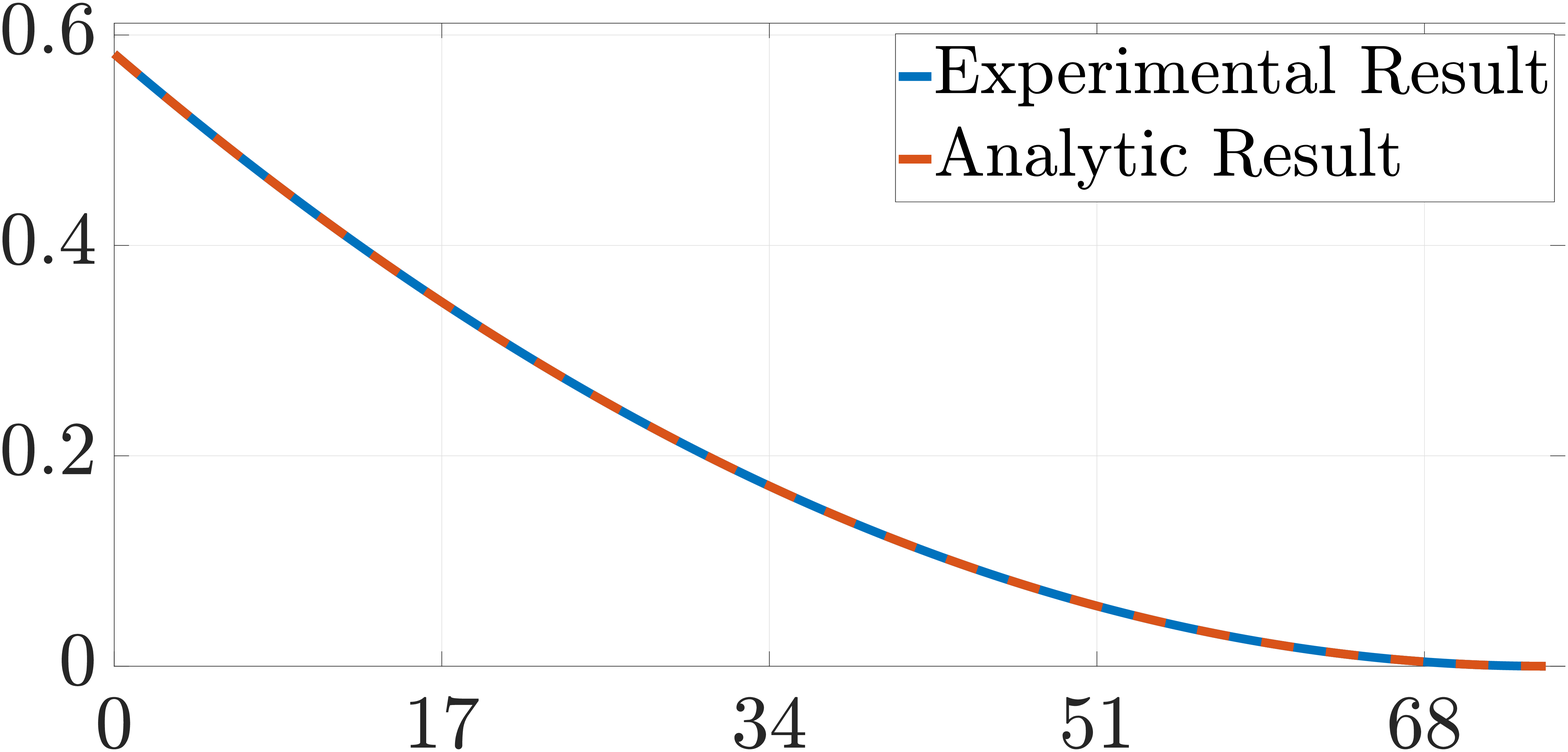}
\label{subfig:decay2D}
}
\caption{{\bf Decay profile of an eigenfunction.} On the left, eigenfunctions of the $p$-Laplacian are shown for $p=1.5$ in one and two dimensions. On the right, the experimental and theoretical values of $u(t,x_0)$ vs. time are shown.}
\label{Fig:1D2D}
\end{figure}
\begin{figure}[htb]
  \centering
  \captionsetup[subfigure]{justification=centering}
  \subfloat[$p=1.3$, $\lambda = -0.0531$]  {\includegraphics[width=0.225\textwidth]{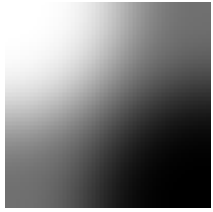}
\label{subfig:EFp1_3}}
$\qquad\qquad\qquad$
\subfloat[$p=1.5$, $\lambda = -0.0269$]
  {
\includegraphics[width=0.225\textwidth]{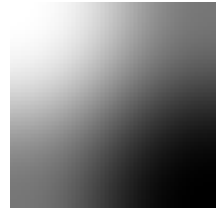}
\label{subfig:EFp1_5}
}
\\
\subfloat[{The spectrum $\abs{S}$ vs. time}]
  {
\includegraphics[width=0.495\textwidth]{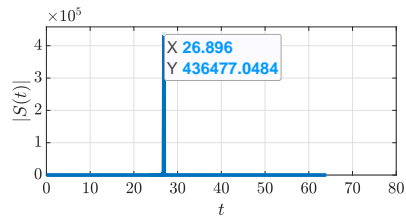}
\label{subfig:specEF1_3}
}
\subfloat[{The spectrum $\abs{S}$ vs. time}]
  {
\includegraphics[width=0.495\textwidth]{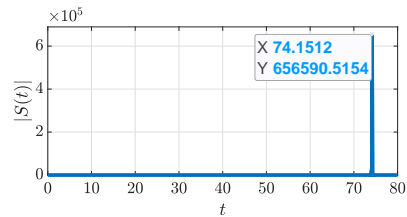}
\label{subfig:specEF1_5}
}
\caption{{\bf $p$-Spectrum of a two dimensional eigenfunction.}
In the first row we show two eigenfunctions of the $p$-Laplace operator for different values of $p$. In the second row we show the absolute value of their $p$-spectra (Eq. \eqref{eq:pSpectrum}). The extinction time of the $p$-flow is $T=26.9$ when $p=1.3$ and $T=74.3$ when $p=1.5$. The delta functions are obtained at the expected time according to Eq. \eqref{eq:gExtinctionTime}.}
\label{fig:EFspectra}
\end{figure}
\begin{figure}[phtb]
\centering
\captionsetup[subfigure]{justification=centering}
\subfloat[Noise]
{
\includegraphics[width=0.2\textwidth]{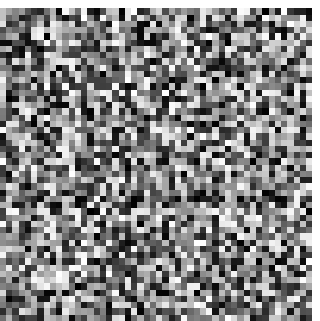}
\label{subfig:noise}
}
\subfloat[$p=1.5$, $\lambda=-0.0269$]
{
\includegraphics[width=0.2\textwidth]{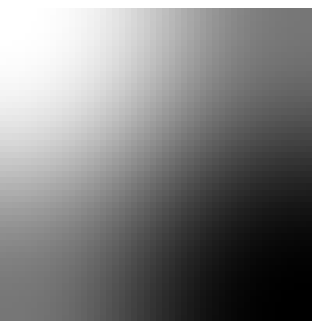}
\label{subfig:EFp1_5sec}
}
\subfloat[ =(a)+(b) The initial condition $f$]
{
\includegraphics[width=0.2\textwidth]{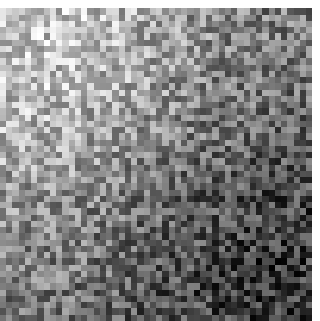}
\label{subfig:EFp1_5pN}
}
\\
\subfloat[{The spectrum $\abs{S}$ vs. time}]
  {
\includegraphics[width=0.5\textwidth,valign=t]{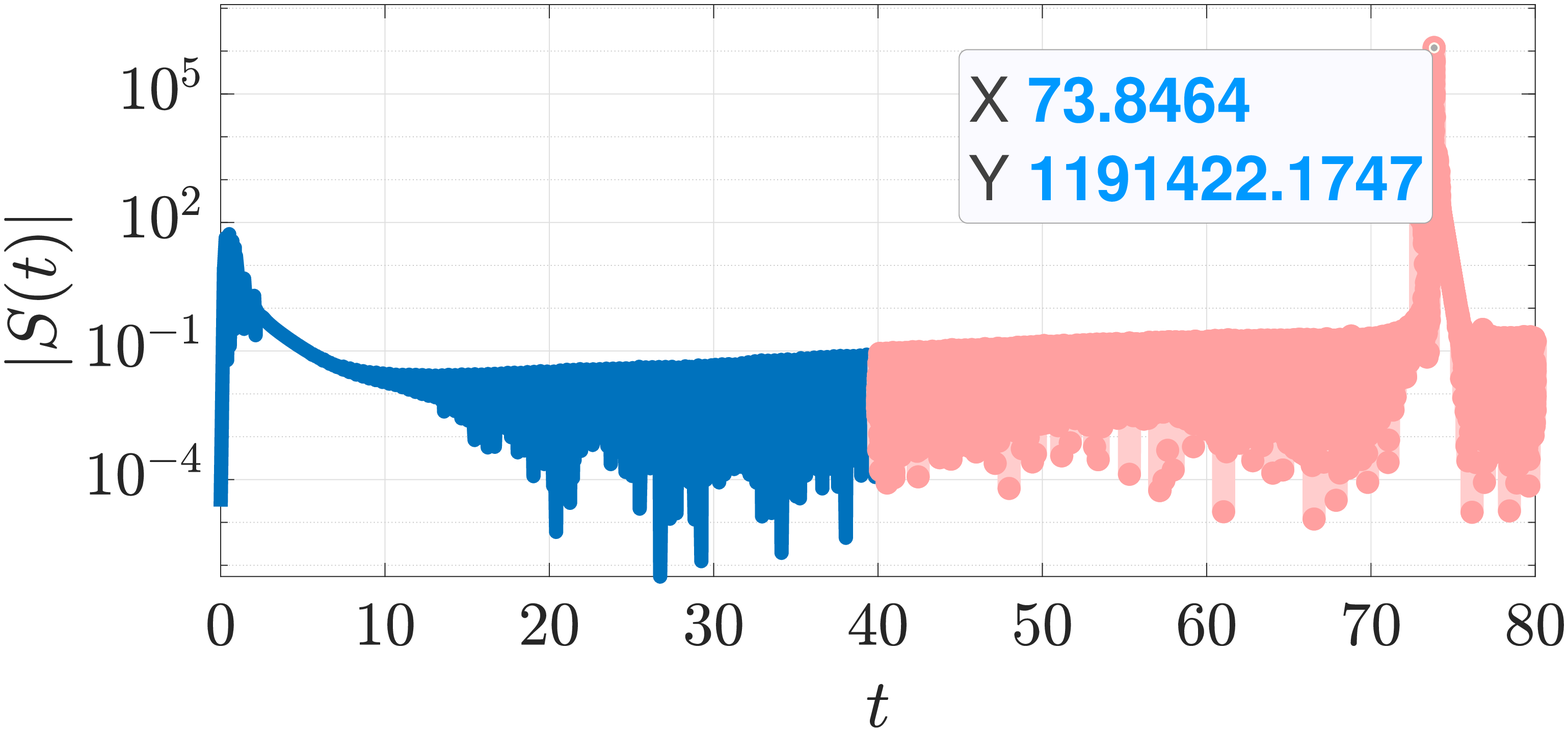}
\label{subfig:specEF1_5pNLog}
}
\subfloat[The filtered out noise]
  {
\includegraphics[width=0.2\textwidth,valign=t]{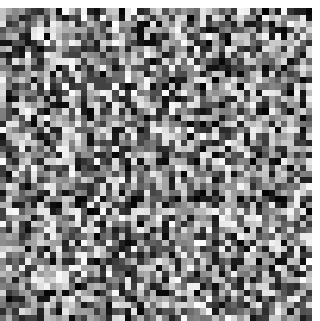}
\label{subfig:filteredNoise}
}
\subfloat[Recovered e.f.]
{
\includegraphics[width=0.2\textwidth,valign=t]{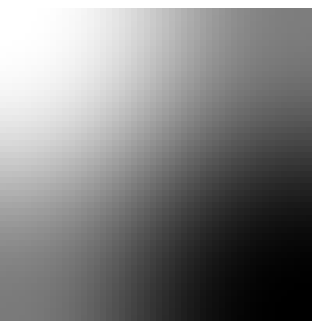}
\label{subfig:filteredEFp1_5}
}
\caption{{\bf Filtering.} Fig. \ref{subfig:noise} is noise image uniformly distributed noise between $[0,1]$. Fig. \ref{subfig:EFp1_5sec} a $p$-Laplacian eigenfunction. Fig. \ref{subfig:EFp1_5pN} is the initial condition in Eq. \eqref{eq:pFlow}. Fig. \ref{subfig:specEF1_5pNLog} is the spectrum of the initial condition where the filtered parts, blue and red, are Figs. \ref{subfig:filteredNoise} and  \ref{subfig:filteredEFp1_5} respectively.}
\label{Fig:filteringEF}
\end{figure}
\begin{figure}[phtb]
\centering
\begin{minipage}{0.20825\textwidth}
\includegraphics[width=\textwidth]{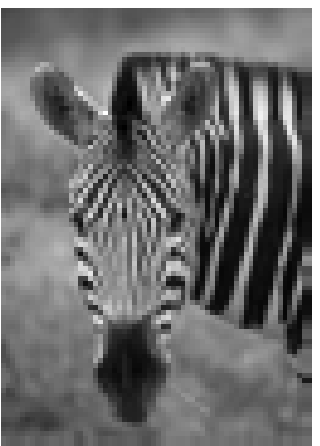}
\caption{A zebra}
\label{Fig:Zebra}
\end{minipage}
\begin{minipage}{0.29175\textwidth}
\includegraphics[width=\textwidth]{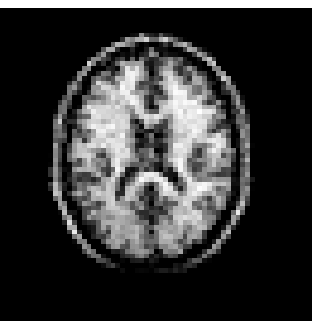}
\caption{A \acrshort{PET} image.}
\label{Fig:PET}
\end{minipage}
\end{figure}
\begin{figure}[phtb]
\centering
\subfloat[The $p$-spectrum $S$ vs. time]
  {
\includegraphics[width=0.6425\textwidth]{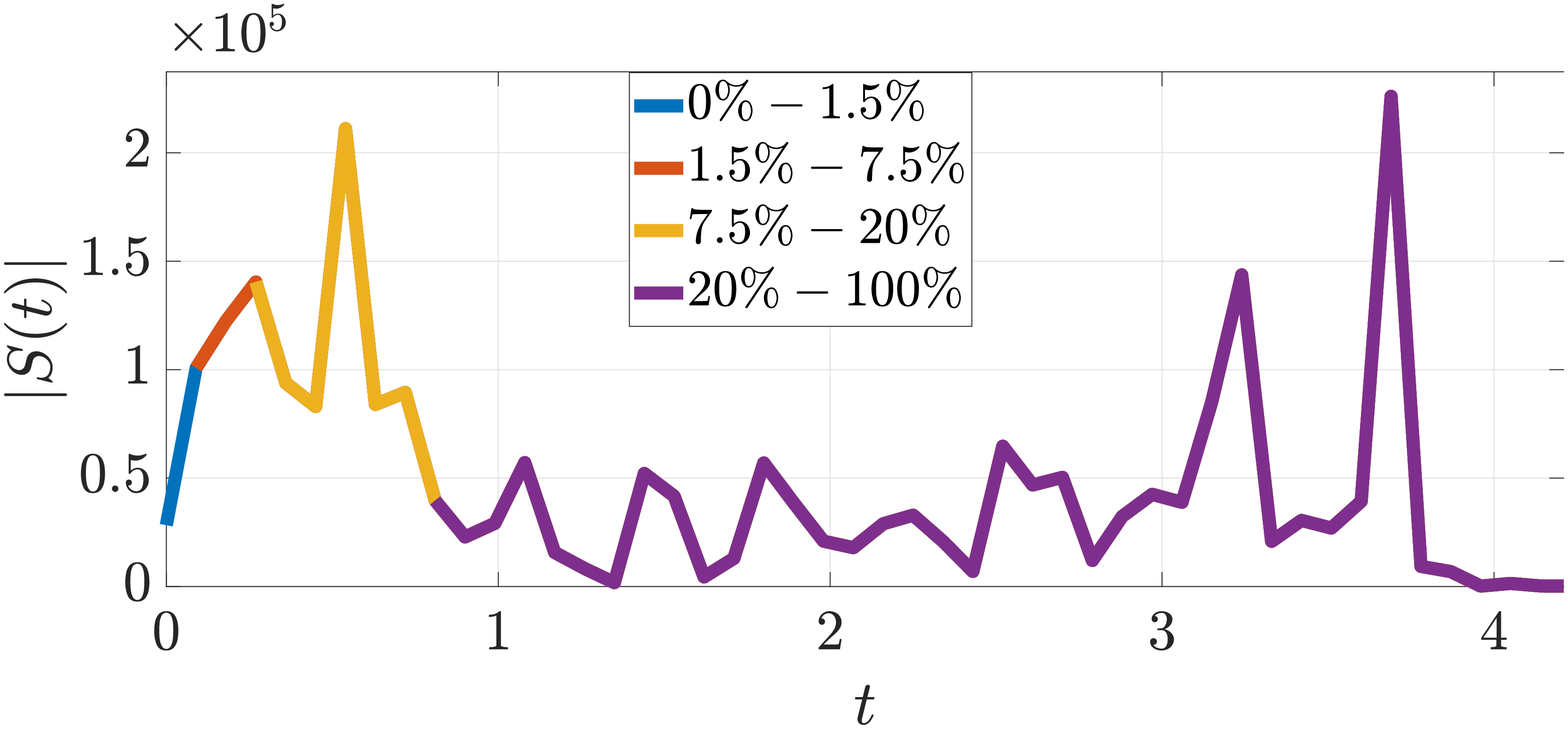}
\label{subfig:specZebra1_01color}
}\\
\subfloat[$0\%-1.5\%$]
  {
\includegraphics[width=0.215\textwidth]{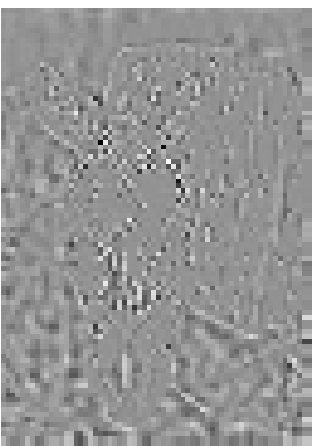}
\label{subfig:z1_101}
}
\subfloat[$1.5\%-7.5\%$]
  {
\includegraphics[width=0.215\textwidth]{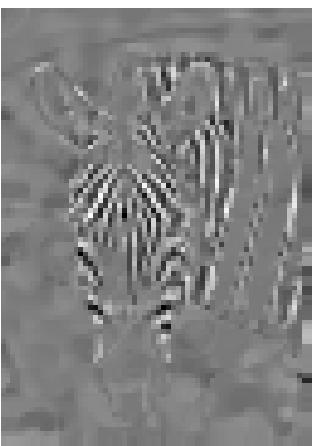}
\label{subfig:z2_101}
}
\subfloat[$7.5\%-20\%$]
  {
\includegraphics[width=0.215\textwidth]{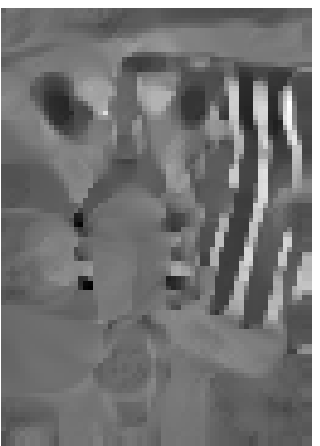}
\label{subfig:z3_101}
}
\subfloat[$20\%-100\%$]
  {
\includegraphics[width=0.215\textwidth]{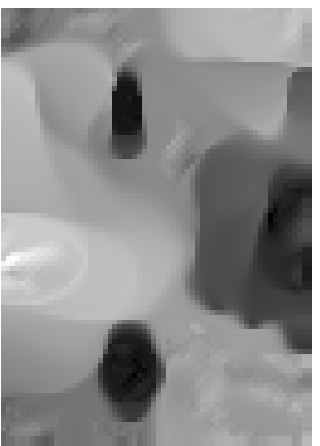}
\label{subfig:z4_101}
}
\caption{Zebra image decomposition with $p=1.01$}
\label{Fig:ZebraDecomposition_101}
\end{figure}
\begin{figure}[phtb]
\centering
\subfloat[The $p$-spectrum $S$ vs. time]
  {
\includegraphics[width=0.6425\textwidth]{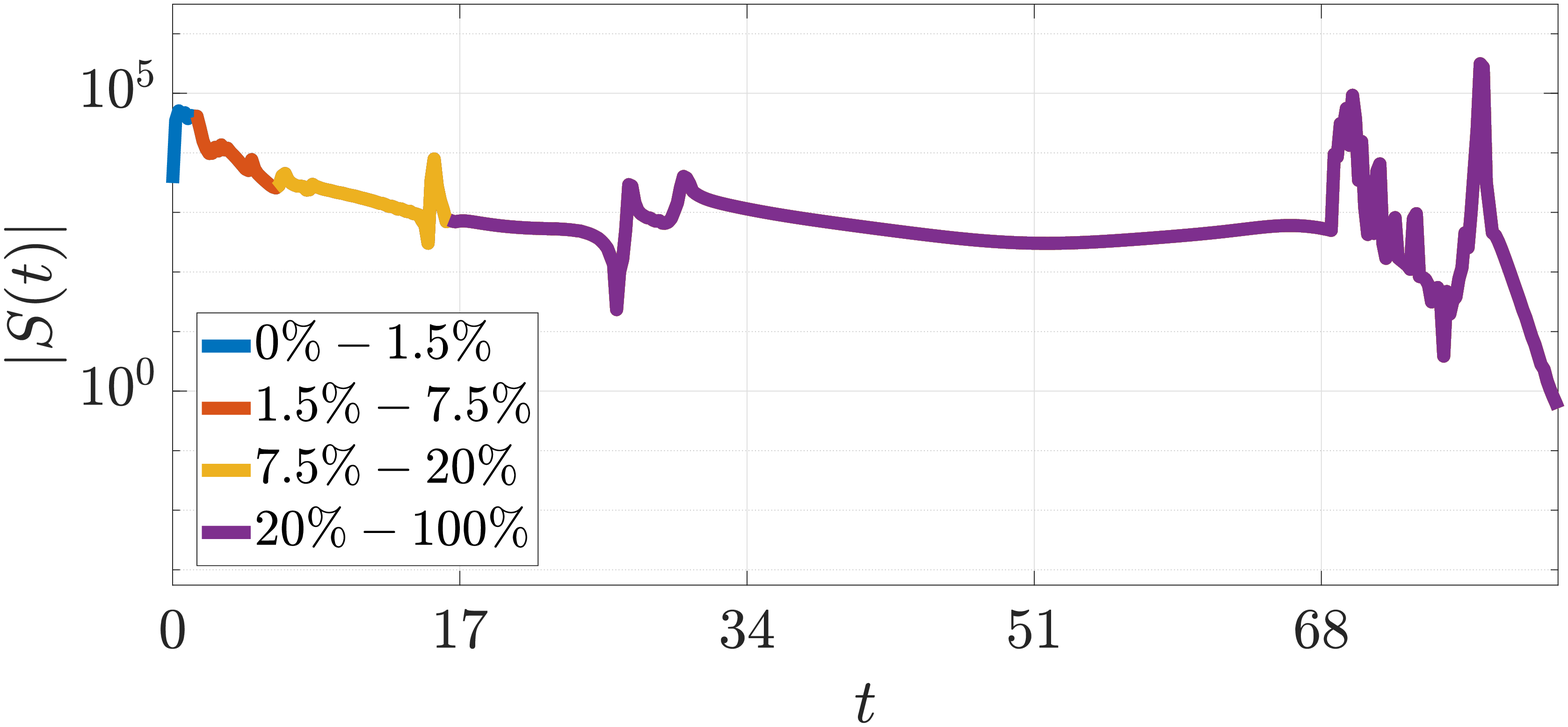}
\label{subfig:specZebra1_5LogColor}
}\\
\subfloat[$0\%-1.5\%$]
  {
\includegraphics[width=0.215\textwidth]{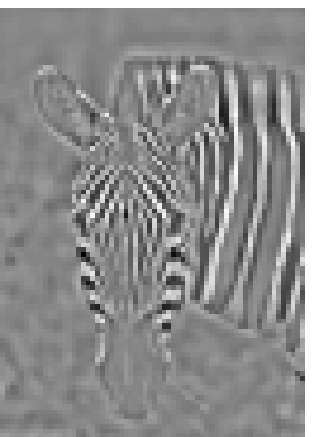}
\label{subfig:z1_15}
}
\subfloat[$1.5\%-7.5\%$]
  {
\includegraphics[width=0.215\textwidth]{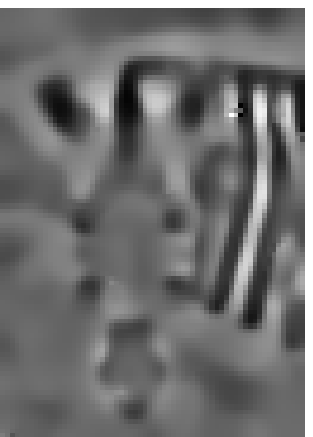}
\label{subfig:z2_15}
}
\subfloat[$7.5\%-20\%$]
  {
\includegraphics[width=0.215\textwidth]{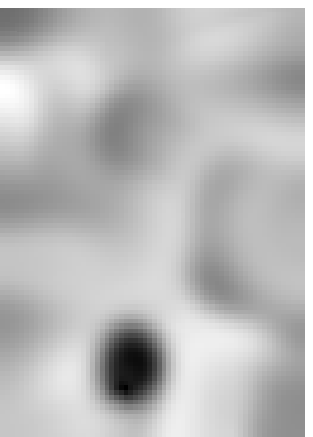}
\label{subfig:z3_15}
}
\subfloat[$20\%-100\%$]
  {
\includegraphics[width=0.215\textwidth]{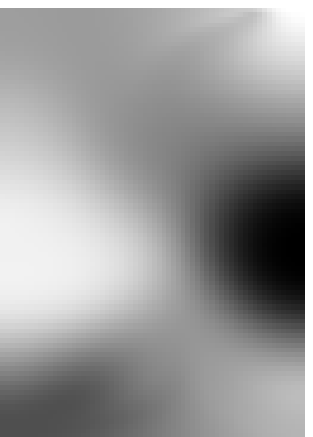}
\label{subfig:z4_15}
}
\caption{Zebra Decomposition with $p=1.5$}
\label{Fig:ZebraDecomposition_15}
\end{figure}
\begin{figure}[phtb]
\centering
\subfloat[The $p$-spectrum $S$]
  {
\includegraphics[width=0.6425\textwidth]{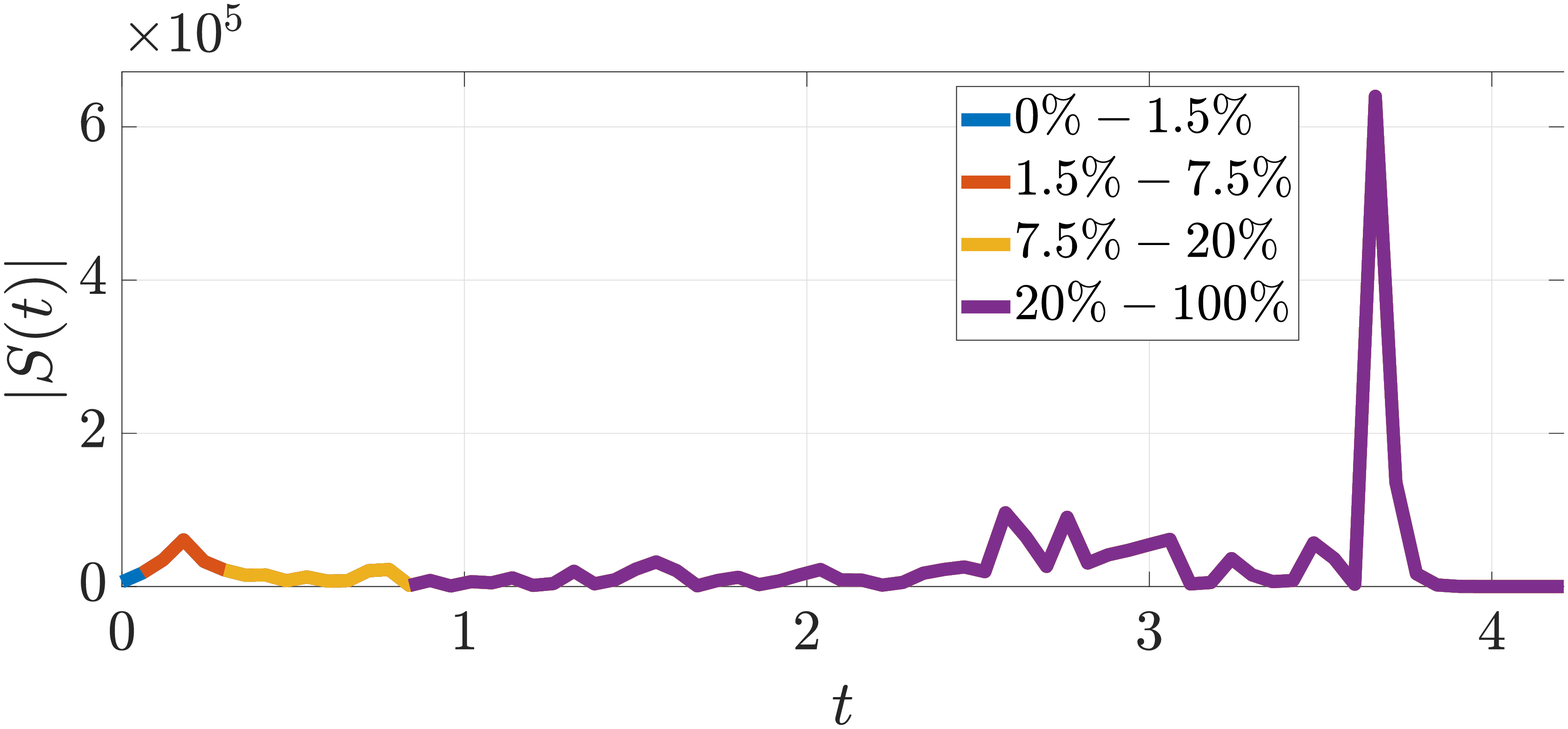}
\label{subfig:specPET1_01Color}
}\\
\subfloat[$0\%-1.5\%$]
  {
\includegraphics[width=0.225\textwidth]{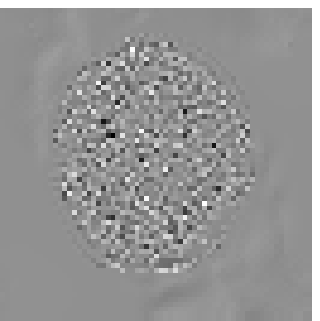}
\label{subfig:f1_101_PET}
}
\subfloat[$1.5\%-7.5\%$]
  {
\includegraphics[width=0.225\textwidth]{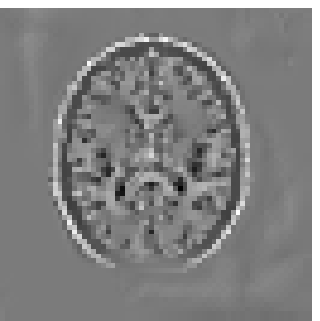}
\label{subfig:f2_101_PET}
}
\subfloat[$7.5\%-20\%$]
  {
\includegraphics[width=0.225\textwidth]{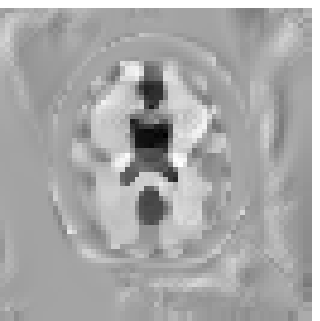}
\label{subfig:f3_101_PET}
}
\subfloat[$20\%-100\%$]
  {
\includegraphics[width=0.225\textwidth]{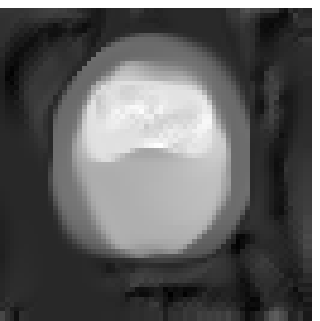}
\label{subfig:f4_101_PET}
}
\caption{A PET image decomposition with $p=1.01$}
\label{Fig:PETDecomposition_101}
\end{figure}
\begin{figure}[phtb]
\centering
\captionsetup[subfigure]{justification=centering}
\subfloat[The $p$-spectrum $S$ vs. time]
  {
\includegraphics[width=0.6425\textwidth]{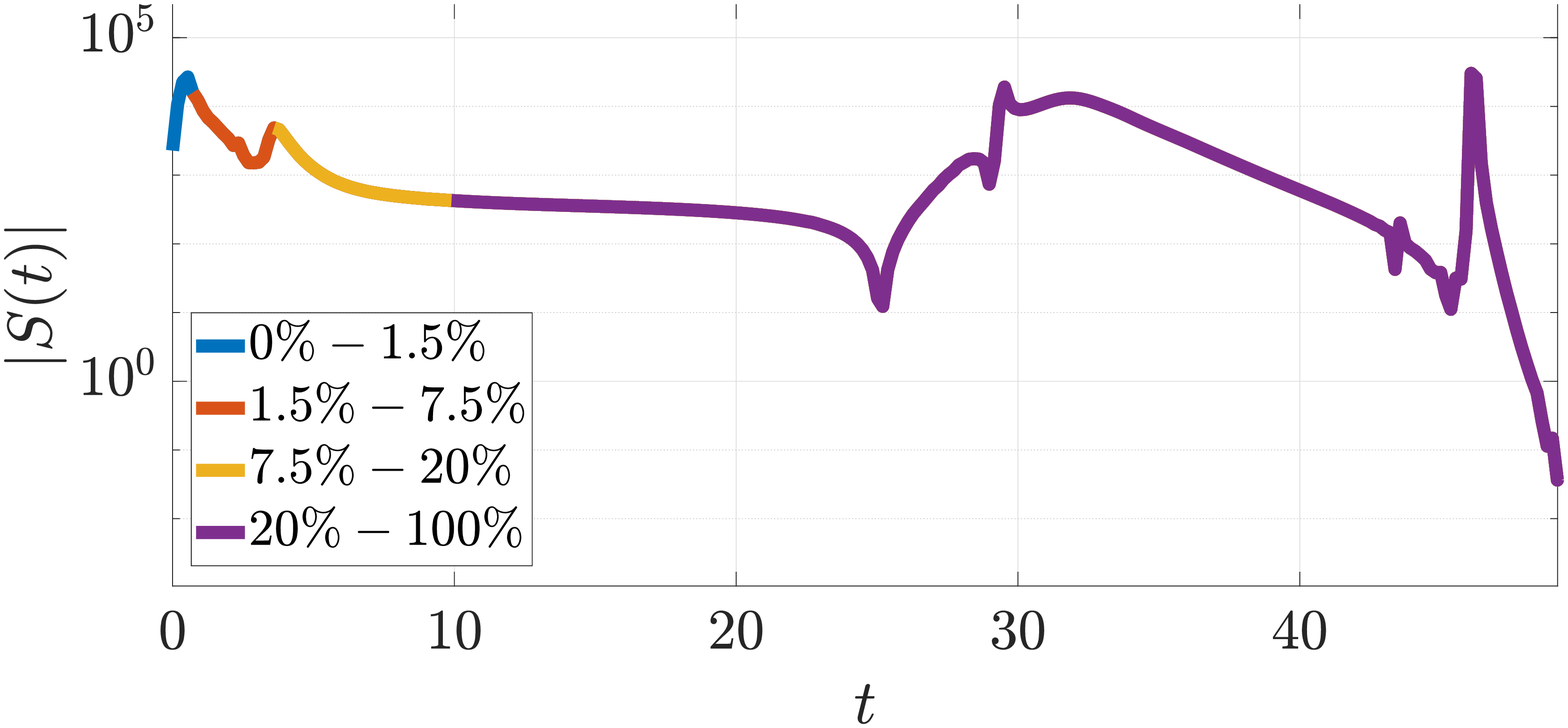}
\label{subfig:specPET1_5LogColor}
}\\
\subfloat[$0\%-1.5\%$]
  {
\includegraphics[width=0.225\textwidth]{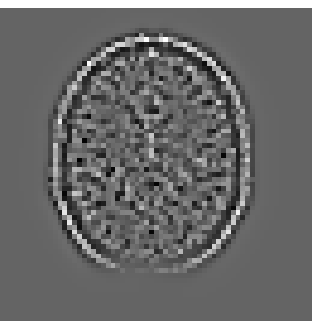}
\label{subfig:f1_15_PET}
}
\subfloat[$1.5\%-7.5\%$]
  {
\includegraphics[width=0.225\textwidth]{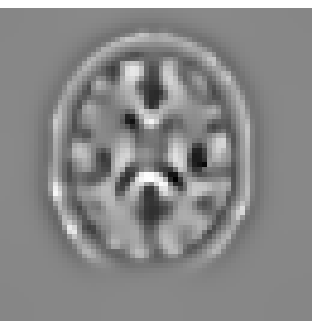}
\label{subfig:f2_15_PET}
}
\subfloat[$7.5\%-20\%$]
  {
\includegraphics[width=0.225\textwidth]{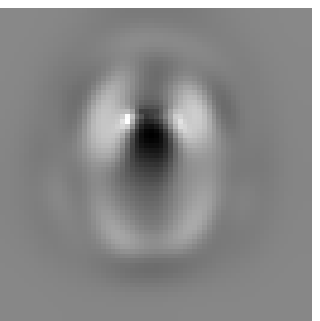}
\label{subfig:f3_15_PET}
}
\subfloat[$20\%-100\%$]
  {
\includegraphics[width=0.225\textwidth]{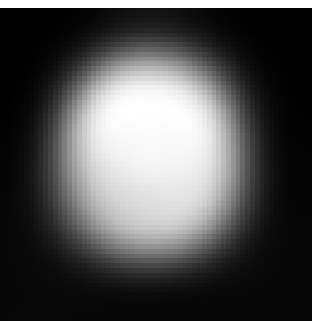}
\label{subfig:f4_15_PET}
}
\caption{A PET image decomposition $p=1.5$}
\label{Fig:PETDecomposition_15}
\end{figure}
\begin{figure}[phtb]
\centering
\captionsetup[subfigure]{justification=centering}
\subfloat[{The zero homogeneous spectrum $S$ "Normalized" vs. time}]
  {
\includegraphics[width=0.6425\textwidth]{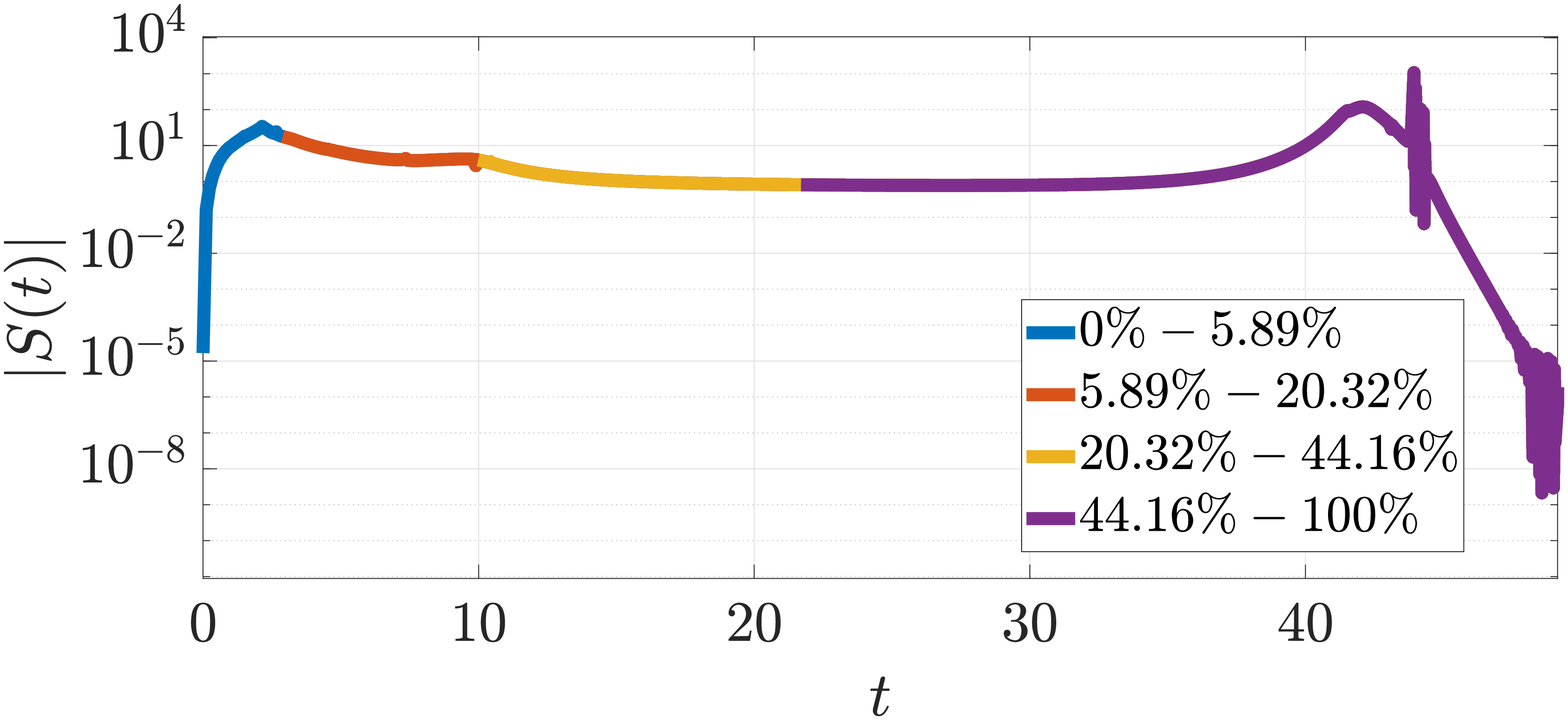}
\label{subfig:specPET1_5LogColor_Normalized}
}\\
\subfloat[$0\%-5.89\%$]
  {
\includegraphics[width=0.225\textwidth]{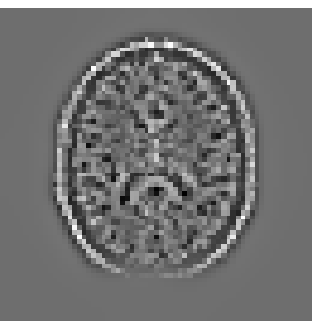}
\label{subfig:f1_15_PET_Normalized}
}
\subfloat[$5.89\%-20.32\%$]
  {
\includegraphics[width=0.225\textwidth]{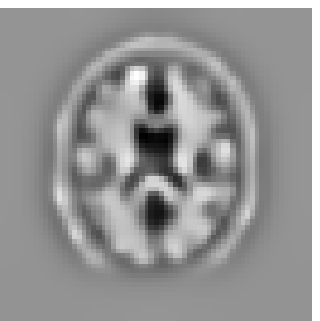}
\label{subfig:f2_15_PET_Normalized}
}
\subfloat[$20.32\%-44.16\%$]
  {
\includegraphics[width=0.225\textwidth]{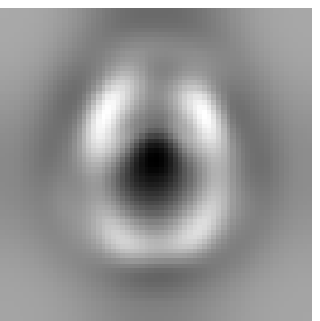}
\label{subfig:f3_15_PET_Normalized}
}
\subfloat[$44.16\%-100\%$]
  {
\includegraphics[width=0.225\textwidth]{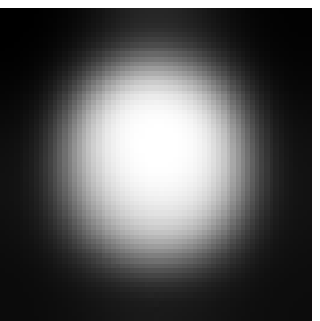}
\label{subfig:f4_15_PET_Normalized}
}
\caption{A PET image decomposition $p=1.5$ with normalized Laplacian operator}
\label{Fig:PETDecomposition_15_Normalized}
\end{figure}
\begin{figure}[phtb]
\centering
\subfloat[The \acrshort{TV}-spectrum vs. time]
  {
\includegraphics[width=0.6425\textwidth]{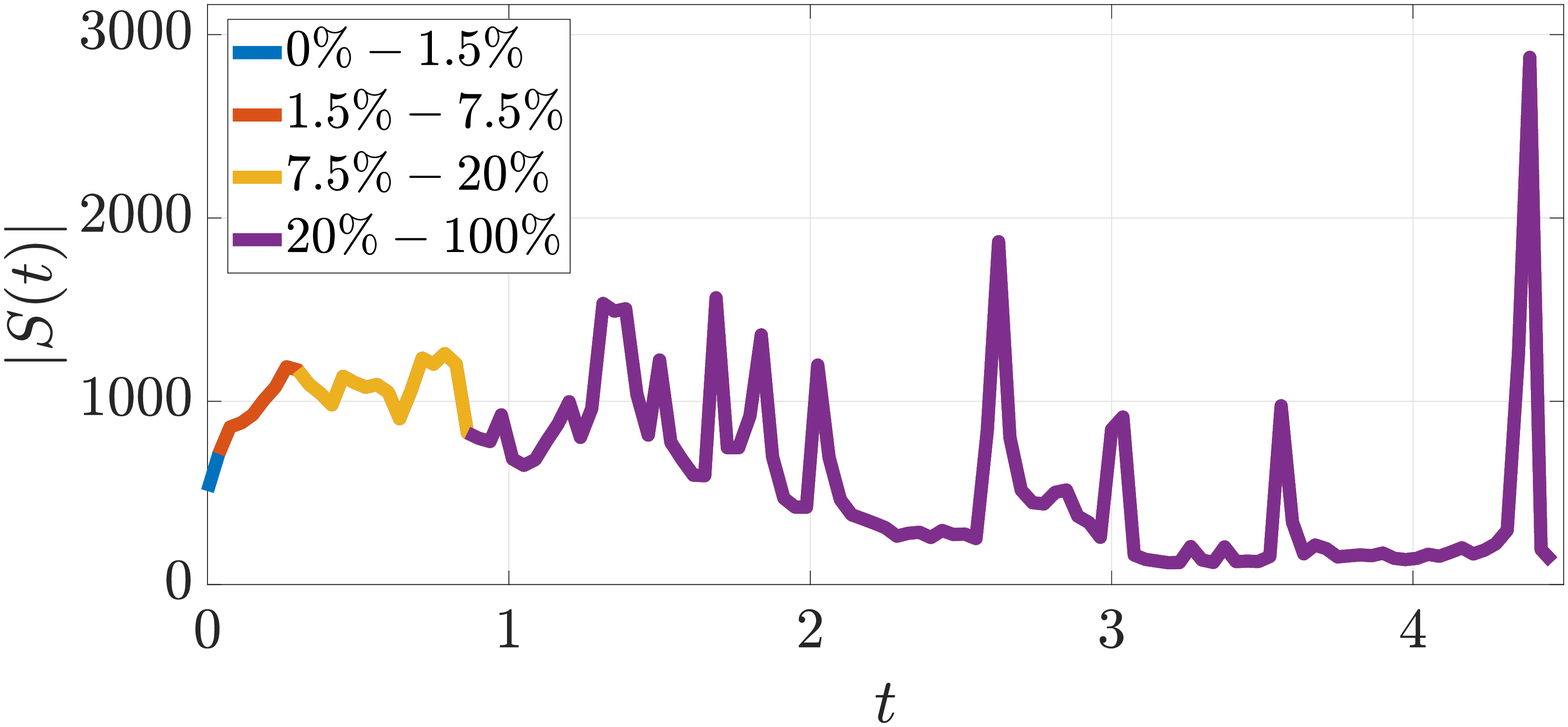}
\label{subfig:TV-zebra}
}\\
\subfloat[$0\%-1.5\%$]
  {
\includegraphics[width=0.215\textwidth]{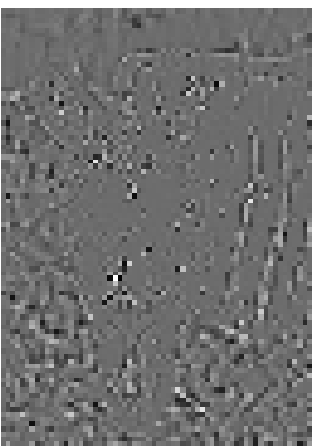}
\label{subfig:z1_TV}
}
\subfloat[$1.5\%-7.5\%$]
  {
\includegraphics[width=0.215\textwidth]{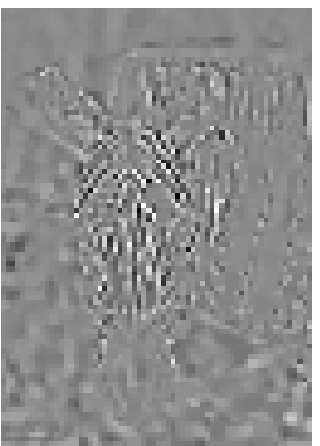}
\label{subfig:z2_TV}
}
\subfloat[$7.5\%-20\%$]
  {
\includegraphics[width=0.215\textwidth]{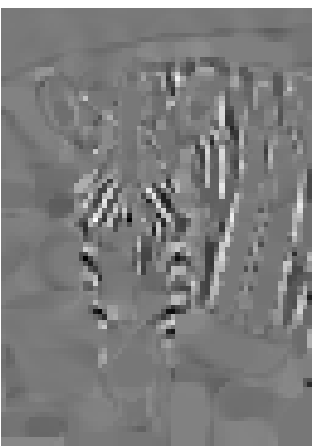}
\label{subfig:z3_TV}
}
\subfloat[$20\%-100\%$]
  {
\includegraphics[width=0.215\textwidth]{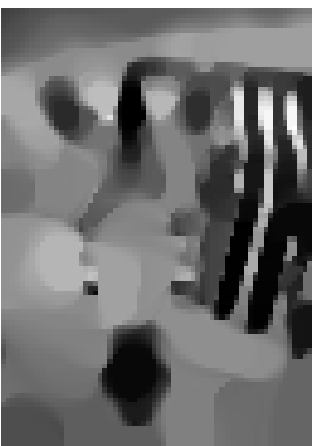}
\label{subfig:z4_TV}
}
\caption{Zebra image \acrshort{TV}-decomposition}
\label{Fig:ZebraDecomposition_TV}
\end{figure}
\section{Experiments}\label{sec:experiments}
In this section we numerically illustrate the theory presented above. We choose the operator $P$ to be the $p$-Laplacian operator where $p\in(1,2)$. We assume finite dimensions and Neumann boundary conditions. This setting admits Assumption \ref{ass:flowAssumption} and therefore all Theorems and Propositions shown earlier are valid.


We use the implementation of the gradient and divergence operator as defined by Chambolle in \cite{chambolle2004algorithm}. We implement our transform in Matlab and some of our experiments can be found \href{https://github.com/IdoCohen5743/pLaplaceFramework.git}{\underline{here}} (https://github.com/IdoCohen5743/pLaplaceFramework.git).

The aims of the following experiments are:
\begin{enumerate}
    \item To validate Theorem \ref{theo:homoOperator} by examining the decay profile of the $p-$flow initiated with an eigenfunction (Fig. \ref{Fig:1D2D}).
    \item  To illustrate Theorem \ref{theo:EFspectrum} by applying the $p$-transform on a single eigenfunction (Fig. \ref{fig:EFspectra}).
    \item To numerically demonstrate the framework of the $p$-transform by filtering and decomposing images with different values of $p$ (Figs. \ref{Fig:filteringEF}, \ref{Fig:ZebraDecomposition_101}, \ref{Fig:PETDecomposition_101}, \ref{Fig:ZebraDecomposition_15}, \ref{Fig:PETDecomposition_15}). 
\end{enumerate} 

The $p$-flow was implemented by an explicit scheme \cite{cohen2019stable} with a fixed time step ($dt=10^{-4}$). We used the fractional derivative implementation of Euler-Gr\"unwald-Letnikov type, as shown in \cite{machado2001discrete}. The eigenfunctions were generated numerically by the algorithm of \cite{cohen2018energy}. 

{\bf Decay profile (Theorem \ref{theo:homoOperator}).} The analytic solution of the decay profile, Eq. \eqref{eq:pFlowAnalytic}, was numerically validated. The analytic solution was computed for a certain spatial coordinate, $x_0$. We show the solutions for $x_0=0$ (in 1D), $x_0 = (0,0)$ (in 2D). A similar behaviour was exhibited for all other points (Fig. \ref{Fig:1D2D}). It can be observed that the analytic solution and the experimental one well agree. In addition, the extinction time was accurately predicted (Eq. \eqref{eq:gExtinctionTime}).

{\bf Spectral behaviour of an eigenfunction (Theorem \ref{theo:EFspectrum}).} Here we illustrate the $p$-spectrum of eigenfunctions with different values of $p$ (Fig. \ref{fig:EFspectra}). On the top two eigenfunctions are shown ($p=1.3$, left, and $p=1.5$, right) for which the $p$-transform is applied to. On the bottom row the magnitude of the spectrum  (Def. \ref{def:pSpectrum}) is presented. One can observe the spectrum has a single dominant scale $t$. Thus, the transform approaches a numerical delta, as predicted by Theorem \ref{theo:EFspectrum}. The theoretical extinction time, Eq. \eqref{eq:gExtinctionTime}, of the eigenfunctions is $26.913$ and $74.298$ for $p=1.3$ and $p=1.5$, respectively. The experimental results show high spectral density around $26.896$ and $74.151$, which agrees well with the theory.

{\bf Filtering.} Filtering of noise is demonstrated in Fig. \ref{Fig:filteringEF}. Ideal LPF (Eq. \eqref{eq:idealLPF}), $t_1=40$, is applied to an image containing an eigenfunction with additive white uniform noise. As can be viewed in the spectral plot, bottom left, the noise corresponds to low values of $t$ (high eigenvalues - blue part) and the spectral part corresponding to the eigenfunction is concentrated at higher scales $t$ (red part). Note, that the peak of the spectrum at high scale (corresponding to the eigenfunction) appears approximately at the same scale as the clean eigenfunction (Fig \ref{subfig:EFp1_5}).

{\bf $p$-Decomposition.} One of the benefits of our framework is the ability to decompose a signal into components with different degree of smoothness, depending on the value of $p$. In this context, a decomposition is a partition of the spectrum into non-overlapping intervals such that the sum of all parts covers the entire spectrum. This is, the sum of all parts is the original image. Here, we decompose two images (Figs. \ref{Fig:Zebra} and \ref{Fig:PET}) by the $p$-transform with two values of $p$ ($p=1.01,\,1.5$). We decompose the spectrum into four parts. As most of the details are concentrated at lower scales $t$, the filters' width is growing with scale. We arbitrarily chose the filters cut offs at $\,1.5\%,\,7.5\%,\,20\%$, where $100\%$ is the width of the entire spectrum. The $p$-decompositions of a zebra image (Fig. \ref{Fig:Zebra}) with $p=1.01$ and $p=1.5$ are shown in Figs. \ref{Fig:ZebraDecomposition_101} and  \ref{Fig:ZebraDecomposition_15}, respectively. The $p$-decompositions of a \acrshort{PET} image (Fig. \ref{Fig:PET}) with $p=1.01$ and $p=1.5$ are shown in Figs. \ref{Fig:PETDecomposition_101} and \ref{Fig:PETDecomposition_15}, respectively.

It can be observed that low scale $t$ is related to fine details in the image when $p=1.01$ and also to edges when $p=1.5$; high scale $t$ is related to the coarser parts of the image. Whereas the $p$-decomposition with $p=1.5$ resembles linear diffusion, the decomposition with $p=1.01$ can be seen as an approximation of the \acrshort{TV}-transform.

{\bf Comparison to TV.} As it well studied, the \acrshort{TV}-flow is implemented with the dual problem (see \cite{chambolle2004algorithm}), which is computationally expensive. However, the $p$-flow can be evaluate explicitly \cite{cohen2019stable}. We can approximate the \acrshort{TV}-transform with the $p$-transform when $p$ is close to one. Theoretically, there is no constrain on the step size of the \acrshort{TV}-flow in a semi-implicit setting. However, to compare between the image spectrum of an image the step size should be similar. The \acrshort{TV}-spectrum, shown in Fig. \ref{subfig:TV-zebra}, is characterized with group of peaks  which are similar to the decomposition with $p=1.01$ Figs. \ref{subfig:specZebra1_01color} and \ref{subfig:specPET1_01Color}.

{\bf Normalized flow.} The last issue we would like to examine is a comparison between the $p$-transform and the one-homogeneous setting from \cite{burger2016spectral}. 
More specifically, we examine the following \emph{normalized flow}
$$u_t(t)=\Delta_p u(t)/\norm{\nabla u(t)}_p^{p-1}.$$
This is a flow based on a zero-homogeneous operator, stemming from an absolutely one-homogeneous functional. 
Theoretically, the $p$-Laplacian flow and the normalized flow are shape preserving flows for the same class of eigenfunctions. The difference is in the decay profiles, where the normalized flow decays linearly. We apply a decomposition with the normalized flow for $p=1.5$. \textcolor{black}{In order to compare between the normalized flow and the original one we should rescale the time axes. The normalized $p$-decomposition (Fig. \ref{Fig:PETDecomposition_15_Normalized}) and the decomposition with $p=1.5$ (Fig. \ref{Fig:PETDecomposition_15}) are very similar but not identical. It is a still an open problem whether it is only a numerical issue. Although this operator is one homogeneous, this decomposition is not similar to \acrshort{TV} decomposition. This kind of decomposition should be further investigated. We know an eigenfunction decays linearly as expected.} We would like to thank the anonymous reviewer for raising this issue.

%
%
%
\section{Conclusion}\label{sec:conc}
The Fourier transform is highly instrumental in processing smooth, band-limited signals. However, filtering signals with inherent discontinuities is known to produce artifacts. The \acrshort{TV}-transform is well adapted for such signals. This work aims to bridge the gap between these two type of transforms, allowing partially smooth signals to be well represented and processed.

Definitions for decomposition, reconstruction, filtering, and spectrum are proposed within a rigorous mathematical framework. In contrast to \cite{bungert2018solution} where the decay profile might be similar to those we study, our decompositions are associated with a much broader class of eigenfunctions, not based only on zero-homogeneous operators. This framework is valid for any functional of homogeneity in the range of $(1,2)$.
Thus, it generalizes the one-homogeneous spectral representation of \cite{burger2016spectral}.

\paragraph{Acknowledgements\textcolor{red}{.}} We acknowledge support by the Israel Science Foundation (grant No. 718/15). This work was supported by the Technion Ollendorff Minerva Center.

\appendix

\section{A proof of Theorem \ref{theo:homoOperator}}\label{sec:appProofTheohomoOperator}
\begin{proof}
$\Leftarrow$ (We assume the solution is shape preserving and prove $f$ is an eigenfunction)\\
\textcolor{black}{
If the solution is shape preserving, since $a(t=0)=1$, we can write
\begin{equation*}
    u_t(t=0)=P(a(t=0)\cdot f)
    =P(f).
\end{equation*}
On the other hand
\begin{equation*}
    u_t(t=0)=a'(t=0)f.
\end{equation*}
By comparing these expressions, we get that $f$ is an eigenfunction with the eigenvalue $\lambda = a'(t=0)$.}\\
$\Rightarrow$ (We assume $f$ is an eigenfunction and we prove the solution is shape preserving)\\
\textcolor{black}{
Let us examine the flow $u(t)=a(t)f$,
where $a(t)$ is defined by the following ODE,
\begin{equation*}
a'(t)= a(t)\abs{a(t)}^{\gamma-1}\lambda,\quad a(0)=1.
\end{equation*}
When $f$ is an eigenfunction with eigenvalue $\lambda$, one can observe this flow solves \eqref{eq:ss}. Existence and uniqueness yield this is the only solution. Eq. \eqref{eq:u_analytic} is the solution for the above ODE.
}
\end{proof}

\section{A proof of Theorem \ref{theo:EFspectrum}}\label{sec:appProofTheoEFspectrum}
\begin{proof}
Let us discuss the monomial $y(x)=x^{\beta}$ where $x\in[0,1]$ and $0$ else. Based on Euler \cite{dalir2010applications,euler1738progressionibus}, we can assert (see proof in \ref{sec:appMonoEuler})
\begin{equation}
    \mathcal{I}_{0^+}^{\ceil{\beta}-\beta}\left\{x^{\beta}\right\}=\frac{\Gamma(\beta+1)}{\Gamma(\ceil{\beta}+1)}x^{\ceil{\beta}}.
\end{equation}
On the other hand, we can say
\begin{equation}
    \begin{split}
        \frac{\Gamma(\beta+1)}{\Gamma(\ceil{\beta}+1)}x^{\ceil{\beta}}=&\frac{1}{\Gamma(\ceil{\beta}-\beta)}\int_0^x \tau^{\beta}(x-\tau)^{\ceil{\beta}-\beta -1}d\tau\\
        \underbrace{=}_{\tau=1-(\gamma -1)\lambda\cdot s}&\frac{1}{\Gamma(\ceil{\beta}-\beta)}\int_{\frac{1}{(\gamma -1)\lambda}}^{\frac{1-x}{(\gamma -1)\lambda}} \left[1-(\gamma -1)\lambda\cdot s\right]^{\beta}\\
        &\quad\left\{x-\left[1-(\gamma -1)\lambda\cdot s\right]\right\}^{\ceil{\beta}-\beta -1}\cdot \left[-(\gamma -1)\lambda\right]\cdot ds\\ 
        =&\frac{(\gamma -1)\lambda}{\Gamma(\ceil{\beta}-\beta)}\int_{\frac{1-x}{(\gamma -1)\lambda}}^{\frac{1}{(\gamma -1)\lambda}} (1-(\gamma -1)\lambda\cdot s)^{\beta}\\
        &\quad\left[(\gamma -1)\lambda\left(\frac{x-1}{(\gamma -1)\lambda}+s\right)\right]^{\ceil{\beta}-\beta -1}ds\\
        =&\frac{\left[(\gamma -1)\lambda\right]^{\ceil{\beta}-\beta}}{\Gamma(\ceil{\beta}-\beta)}\int_{\frac{1-x}{(\gamma -1)\lambda}}^{\frac{1}{(\gamma -1)\lambda}} (1-(\gamma -1)\lambda\cdot s)^{\beta}\left[\frac{x-1}{(\gamma -1)\lambda}+s\right]^{\ceil{\beta}-\beta -1}ds.
            \end{split}
\end{equation}
Let us denote $t=\frac{1-x}{(\gamma -1)\lambda}$ and we have
\begin{equation}
    \begin{split}
        \frac{\Gamma(\beta+1)}{\Gamma(\ceil{\beta}+1)}&\left\{\left[1+(1-\gamma)\lambda t\right]^+\right\}^{\ceil{\beta}}=\\
        =&\frac{\left[(\gamma -1)\lambda\right]^{\ceil{\beta}-\beta}}{\Gamma(\ceil{\beta}-\beta)}
        \int_{t}^{\frac{1}{(\gamma -1)\lambda}} \left[1-(\gamma -1)\lambda\cdot s\right]^{\beta}\left[s-t\right]^{\ceil{\beta}-\beta -1}ds\\
        =&\left[(\gamma -1)\lambda\right]^{\ceil{\beta}-\beta} I_{{\frac{1}{(\gamma -1)\lambda}}^-}^{\ceil{\beta}-\beta}\left\{\left[\left(1+(1-\gamma)\lambda\cdot t\right)^+\right]^{\beta}\right\}.
    \end{split}
\end{equation}
Then
\begin{equation}\label{eq:intermidiatEq}
    \begin{split}
        \frac{\Gamma(\beta+1)}{\Gamma(\ceil{\beta}+1)}\left[\left(1+(1-\gamma)\lambda t\right)^+\right]^{\ceil{\beta}}=&\left[(\gamma -1)\lambda\right]^{\ceil{\beta}-\beta} I_{{\frac{1}{(\gamma -1)\lambda}}^-}^{\ceil{\beta}-\beta}\left\{\left[\left(1+(1-\gamma)\lambda\cdot t\right)^+\right]^{\beta}\right\}
    \end{split}
\end{equation}
Now, we apply the operator ${d^{\ceil{\beta}+1}}/{dt^{\ceil{\beta}+1}}$. For convenience we separate the discussion to the right and the left hand side of \eqref{eq:intermidiatEq}. Applying the derivative operator on the right hand side gives:
\begin{equation*}
    \begin{split}
        \frac{d^{\ceil{\beta}+1}}{dt^{\ceil{\beta}+1}}\left\{\frac{\Gamma(\beta+1)}{\Gamma(\ceil{\beta}+1)}\left[\left(1+(1-\gamma)\lambda t\right)^+\right]^{\ceil{\beta}}\right\}=\\
        =\Gamma(\beta+1)\left[(1-\gamma)\lambda\right]^{\ceil{\beta}+1}\delta(1+(1-\gamma)\lambda t).
    \end{split}
\end{equation*}
Applying the derivative operator on the right hand side and using \eqref{eq:LiouvilleDerivative-} gives:
\begin{equation*}
    \begin{split}
    \frac{d^{\ceil{\beta}+1}}{dt^{\ceil{\beta}+1}}\left\{{\left[(\gamma -1)\lambda\right]^{\ceil{\beta}-\beta}} I_{{\frac{1}{(\gamma -1)\lambda}}^-}^{\ceil{\beta}-\beta}\left\{\left[\left(1+(1-\gamma)\lambda t\right)^+\right]^{\beta}\right\}\right\}
    =\\
    =\left[(\gamma -1)\lambda\right]^{\ceil{\beta}-\beta}(-1)^{\ceil{\beta}+1} D_{{\frac{1}{(\gamma -1)\lambda}}^-}^{\beta +1}\left\{\left[\left(1+(1-\gamma)\lambda\cdot t\right)^+\right]^{\beta}\right\}.
    \end{split}
\end{equation*}
Comparing these expressions yields
\begin{equation*}
    \begin{split}
        D_{{\frac{1}{(\gamma -1)\lambda}}^-}^{\beta+1}\left\{\left[\left(1+(1-\gamma)\lambda\cdot t\right)^+\right]^{\beta}\right\}=&\Gamma(\beta+1)\left[(\gamma -1)\lambda\right]^{\beta+1}\delta(1+(1-\gamma)\lambda t).
    \end{split}
\end{equation*}
Substituting it in \eqref{eq:pTransform} we reach \eqref{eq:ptransformEF}.
\end{proof}
Note, we could choose any value greater than $\frac{1}{(\gamma -1)\lambda}$ (the extinction time of the flow) since $u(t) = 0,\,\forall t>T$.


\section{Fractional integration of a monomial function}\label{sec:appMonoEuler}
Let the function $f$ be
\begin{equation*}
f(x)=
    \begin{cases}
    x^\beta&x\in[0,1]\\
    0&else
    \end{cases}
\end{equation*}
where $\beta\in\mathbb{R}_{>0}$. The $\ceil{\beta}-\beta$ order fractional integral of $f$ is
\begin{equation*}
    I_{0^+}^{\ceil{\beta}-\beta}f(x)=
    \begin{cases}
    0&x<0\\
    \frac{1}{\Gamma\left(\ceil{\beta}-\beta\right)}\int_{0}^x(x-\tau)^{\ceil{\beta}-\beta -1}\tau^\beta d\tau&x\in[0,1]\\
    \frac{1}{\Gamma\left(\ceil{\beta}-\beta\right)}\int_{0}^1(x-\tau)^{\ceil{\beta}-\beta -1}\tau^\beta d\tau&x>1
    \end{cases}
\end{equation*}
Let us simplify the following expression
\begin{equation*}
    \begin{split}
        \int_{0}^x(x-\tau)^{\ceil{\beta}-\beta -1}\tau^\beta d\tau=&x^{\ceil{\beta}-\beta -1}\int_{0}^x\left(1-\frac{\tau}{x}\right)^{\ceil{\beta}-\beta -1}\tau^\beta d\tau\\
        =&x^{\ceil{\beta}-1}\int_{0}^x\left(1-\frac{\tau}{x}\right)^{\ceil{\beta}-\beta -1}\left(\frac{\tau}{x}\right)^\beta d\tau\\
        \underbrace{=}_{s=\frac{\tau}{x}}&x^{\ceil{\beta}}\underbrace{\int_{0}^1\left(1-s\right)^{\ceil{\beta}-\beta -1}s^\beta ds}_{ \textrm{beta function}}\\
        =&x^{\ceil{\beta}}\frac{\Gamma(\ceil{\beta}-\beta)\Gamma(\beta+1)}{\Gamma(\ceil{\beta}+1)}.
    \end{split}
\end{equation*}
And we get
\begin{equation*}
    I_{0^+}^{\ceil{\beta}-\beta}f(x)=
    \begin{cases}
    0&x<0\\
    \frac{\Gamma(\beta+1)}{\Gamma(\ceil{\beta}+1)}x^{\ceil{\beta}}&x\in[0,1]\\
    \frac{\Gamma(\beta+1)}{\Gamma(\ceil{\beta}+1)}&x>1
    \end{cases}
\end{equation*}

We recall here the Euler generalization of integer derivative of monomial \cite{dalir2010applications,euler1738progressionibus}
\begin{equation}
    D^{m}\left\{x^n\right\}=\frac{\Gamma(n+1)}{\Gamma(n-m+1)}x^{n-m}.
\end{equation}


\bibliography{smartPeople}

\begin{thebibliography}{10}
\expandafter\ifx\csname url\endcsname\relax
  \def\url#1{\texttt{#1}}\fi
\expandafter\ifx\csname urlprefix\endcsname\relax\def\urlprefix{URL }\fi
\expandafter\ifx\csname href\endcsname\relax
  \def\href#1#2{#2} \def\path#1{#1}\fi

\bibitem{fourier1822theorie}
J.~Fourier, Theorie analytique de la chaleur, par M. Fourier, Chez Firmin
  Didot, p{\`e}re et fils, 1822.

\bibitem{gilboa2013spectral}
G.~Gilboa, A spectral approach to total variation, in: International Conference
  on Scale Space and Variational Methods in Computer Vision, Springer, 2013,
  pp. 36--47.

\bibitem{gilboa2014total}
G.~Gilboa, A total variation spectral framework for scale and texture analysis,
  SIAM journal on Imaging Sciences 7~(4) (2014) 1937--1961.
\newblock \href {http://dx.doi.org/10.1137/130930704}
  {\path{doi:10.1137/130930704}}.

\bibitem{andreu2001minimizing}
F.~Andreu, C.~Ballester, V.~Caselles, J.~M. Maz{\'o}n, et~al., Minimizing total
  variation flow, Differential and integral equations 14~(3) (2001) 321--360.

\bibitem{garcia1987existence}
J.~Garc{\'\i}a~Azorero, I.~Peral~Alonso, Existence and nonuniqueness for the
  p-laplacian, Communications in Partial Differential Equations 12~(12) (1987)
  126--202.
\newblock \href {http://dx.doi.org/10.1080/03605308708820534}
  {\path{doi:10.1080/03605308708820534}}.

\bibitem{cohen2018energy}
I.~Cohen, G.~Gilboa, Energy dissipating flows for solving nonlinear eigenpair
  problems, Journal of Computational Physics 375 (2018) 1138--1158.
\newblock \href {http://dx.doi.org/10.1016/j.jcp.2018.09.012}
  {\path{doi:10.1016/j.jcp.2018.09.012}}.

\bibitem{barenblatt1952self}
G.~I. Barenblatt, On self-similar motions of a compressible fluid in a porous
  medium, Akad. Nauk SSSR. Prikl. Mat. Meh 16~(6) (1952) 79--6.

\bibitem{kamin1988fundamental}
S.~Kamin, J.~L. V{\'a}zquez, Fundamental solutions and asymptotic behaviour for
  the $p$-laplacian equation, Revista Matem{\'a}tica Iberoamericana 4~(2)
  (1988) 339--354.

\bibitem{liu2016renormalized}
Q.~Liu, Z.~Guo, C.~Wang, Renormalized solutions to a reaction-diffusion system
  applied to image denoising, Discrete \& Continuous Dynamical Systems-B 21~(6)
  (2016) 1839--1858.
\newblock \href {http://dx.doi.org/10.3934/dcdsb.2016025}
  {\path{doi:10.3934/dcdsb.2016025}}.

\bibitem{chen2006variable}
Y.~Chen, S.~Levine, M.~Rao, Variable exponent, linear growth functionals in
  image restoration, SIAM journal on Applied Mathematics 66~(4) (2006)
  1383--1406.
\newblock \href {http://dx.doi.org/10.1137/050624522}
  {\path{doi:10.1137/050624522}}.

\bibitem{baravdish2015backward}
G.~Baravdish, O.~Svensson, F.~{\AA}str{\"o}m, On backward p (x)-parabolic
  equations for image enhancement, Numerical Functional Analysis and
  Optimization 36~(2) (2015) 147--168.
\newblock \href {http://dx.doi.org/10.1080/01630563.2014.970643}
  {\path{doi:10.1080/01630563.2014.970643}}.

\bibitem{huang2016level}
C.~Huang, L.~Zeng, Level set evolution model for image segmentation based on
  variable exponent p-laplace equation, Applied Mathematical Modelling
  40~(17-18) (2016) 7739--7750.
\newblock \href {http://dx.doi.org/10.1016/j.apm.2016.03.039}
  {\path{doi:10.1016/j.apm.2016.03.039}}.

\bibitem{chen2010image}
J.~Chen, J.~Guo, Image restoration based on adaptive p-laplace diffusion, in:
  Image and Signal Processing (CISP), 2010 3rd International Congress on,
  Vol.~1, IEEE, 2010, pp. 143--146.
\newblock \href {http://dx.doi.org/10.1109/CISP.2010.5646369}
  {\path{doi:10.1109/CISP.2010.5646369}}.

\bibitem{yi2017variable}
Z.~Yi, Y.~Ge, A variable exponent p-laplace variational model preserving
  texture for image interpolation, in: Applications of Computer Vision
  Workshops (WACVW), 2017 IEEE Winter, IEEE, 2017, pp. 36--41.
\newblock \href {http://dx.doi.org/10.1109/WACVW.2017.13}
  {\path{doi:10.1109/WACVW.2017.13}}.

\bibitem{maiseli2015multi}
B.~J. Maiseli, O.~A. Elisha, H.~Gao, A multi-frame super-resolution method
  based on the variable-exponent nonlinear diffusion regularizer, EURASIP
  Journal on Image and Video Processing 2015~(1) (2015) 22.
\newblock \href {http://dx.doi.org/10.1186/s13640-015-0077-2}
  {\path{doi:10.1186/s13640-015-0077-2}}.

\bibitem{baravdish2018damped}
G.~Baravdish, O.~Svensson, M.~Gulliksson, Y.~Zhang, A damped flow for image
  denoising, arXiv preprint arXiv:1806.06732.

\bibitem{wei2012p}
W.~Wei, B.~Zhou, A p-laplace equation model for image denoising, Inform.
  Technol. J 11 (2012) 632--636.
\newblock \href {http://dx.doi.org/10.3923/itj.2012.632.636}
  {\path{doi:10.3923/itj.2012.632.636}}.

\bibitem{kuijper2007p}
A.~Kuijper, p-laplacian driven image processing, in: Image Processing, 2007.
  ICIP 2007. IEEE International Conference on, Vol.~5, IEEE, 2007, pp. V--257.
\newblock \href {http://dx.doi.org/10.1109/ICIP.2007.4379814}
  {\path{doi:10.1109/ICIP.2007.4379814}}.

\bibitem{kuijper2013image}
A.~Kuijper, Image processing by minimising l p norms, Pattern recognition and
  image analysis 23~(2) (2013) 226--235.
\newblock \href {http://dx.doi.org/10.1134/S105466181302003X}
  {\path{doi:10.1134/S105466181302003X}}.

\bibitem{kuijper2009geometrical}
A.~Kuijper, Geometrical pdes based on second-order derivatives of gauge
  coordinates in image processing, Image and Vision Computing 27~(8) (2009)
  1023--1034.
\newblock \href {http://dx.doi.org/10.1016/j.imavis.2008.09.003}
  {\path{doi:10.1016/j.imavis.2008.09.003}}.

\bibitem{blomgren1997total}
P.~Blomgren, T.~F. Chan, P.~Mulet, C.-K. Wong, Total variation image
  restoration: numerical methods and extensions, in: Image Processing, 1997.
  Proceedings., International Conference on, Vol.~3, IEEE, 1997, pp. 384--387.
\newblock \href {http://dx.doi.org/10.1109/ICIP.1997.632128}
  {\path{doi:10.1109/ICIP.1997.632128}}.

\bibitem{calder2018game}
J.~Calder, The game theoretic p-laplacian and semi-supervised learning with few
  labels, Nonlinearity 32~(1) (2018) 301.

\bibitem{flores2018algorithms}
M.~Flores, Algorithms for semisupervised learning on graphs, Ph.D. thesis,
  University of Minnesota (2018).

\bibitem{liu2018p}
W.~Liu, X.~Ma, Y.~Zhou, D.~Tao, J.~Cheng, p-laplacian regularization for scene
  recognition, IEEE transactions on cybernetics~(99) (2018) 1--14.

\bibitem{ma2018hypergraph}
X.~Ma, W.~Liu, S.~Li, D.~Tao, Y.~Zhou, Hypergraph $ p $-laplacian
  regularization for remotely sensed image recognition, IEEE Transactions on
  Geoscience and Remote Sensing 57~(3) (2018) 1585--1595.

\bibitem{burger2016spectral}
M.~Burger, G.~Gilboa, M.~Moeller, L.~Eckardt, D.~Cremers, Spectral
  decompositions using one-homogeneous functionals, SIAM Journal on Imaging
  Sciences 9~(3) (2016) 1374--1408.
\newblock \href {http://dx.doi.org/10.1137/15M1054687}
  {\path{doi:10.1137/15M1054687}}.

\bibitem{brezis1973ope}
H.~Brezis, Op\'erateurs maximaux monotones et semi-groupes de contractions dans
  les espaces de Hilbert, Vol.~5, Elsevier, 1973.

\bibitem{dalir2010applications}
M.~Dalir, M.~Bashour, Applications of fractional calculus, Applied Mathematical
  Sciences 4~(21) (2010) 1021--1032.

\bibitem{grigoletto2013fractional}
E.~C. Grigoletto, E.~C. de~Oliveira, Fractional versions of the fundamental
  theorem of calculus, Applied Mathematics 4~(07) (2013) 23.

\bibitem{richard2014fractional}
H.~Richard, Fractional calculus: an introduction for physicists, World
  Scientific, 2014.

\bibitem{kilbas2006theory}
A.~A. Kilbas, H.~M. Srivastava, J.~J. Trujillo, Theory and applications of
  fractional differential equations, Vol. 204, Elsevier Science Limited, 2006.

\bibitem{samko1993fractional}
S.~G. Samko, A.~A. Kilbas, O.~I. Marichev, Fractional integrals and
  derivatives: theory and applications.

\bibitem{bungert2019asymptotic}
L.~Bungert, M.~Burger, Asymptotic profiles of nonlinear homogeneous evolution
  equations of gradient flow type, arXiv preprint arXiv:1906.09856.

\bibitem{bellettini2002total}
G.~Bellettini, V.~Caselles, M.~Novaga, The total variation flow in rn, Journal
  of Differential Equations 184~(2) (2002) 475--525.
\newblock \href {http://dx.doi.org/10.1006/jdeq.2001.4150}
  {\path{doi:10.1006/jdeq.2001.4150}}.

\bibitem{cohen:hal-01870019}
I.~Cohen, G.~Gilboa,
  \href{https://hal.archives-ouvertes.fr/hal-01870019}{{Shape Preserving Flows
  and the p--Laplacian Spectra}}, working paper or preprint (Oct. 2018).
\newline\urlprefix\url{https://hal.archives-ouvertes.fr/hal-01870019}

\bibitem{bungert2019nonlinear}
L.~Bungert, M.~Burger, A.~Chambolle, M.~Novaga, Nonlinear spectral
  decompositions by gradient flows of one-homogeneous functionals, arXiv
  preprint arXiv:1901.06979.

\bibitem{schmidt2018inverse}
M.~F. Schmidt, M.~Benning, C.-B. Sch{\"o}nlieb, Inverse scale space
  decomposition, Inverse Problems 34~(4) (2018) 045008.

\bibitem{moeller2015learning}
M.~Moeller, J.~Diebold, G.~Gilboa, D.~Cremers, Learning nonlinear spectral
  filters for color image reconstruction, in: Proceedings of the IEEE
  International Conference on Computer Vision, 2015, pp. 289--297.
\newblock \href {http://dx.doi.org/10.1109/ICCV.2015.41}
  {\path{doi:10.1109/ICCV.2015.41}}.

\bibitem{zeune2017multiscale}
L.~Zeune, G.~van Dalum, L.~W. Terstappen, S.~A. van Gils, C.~Brune, Multiscale
  segmentation via bregman distances and nonlinear spectral analysis, SIAM
  journal on imaging sciences 10~(1) (2017) 111--146.
\newblock \href {http://dx.doi.org/10.1137/16M1074503}
  {\path{doi:10.1137/16M1074503}}.

\bibitem{hait2018spectral}
E.~Hait, G.~Gilboa, Spectral total-variation local scale signatures for image
  manipulation and fusion\href {http://dx.doi.org/10.1109/TIP.2018.2872630}
  {\path{doi:10.1109/TIP.2018.2872630}}.

\bibitem{chambolle2004algorithm}
A.~Chambolle, An algorithm for total variation minimization and applications,
  Journal of Mathematical imaging and vision 20~(1-2) (2004) 89--97.

\bibitem{cohen2019stable}
I.~Cohen, A.~Falik, G.~Gilboa, Stable explicit p-laplacian flows based on
  nonlinear eigenvalue analysis, in: International Conference on Scale Space
  and Variational Methods in Computer Vision, Springer, 2019, pp. 315--327.

\bibitem{machado2001discrete}
J.~Machado, Discrete-time fractional-order controllers, Fractional Calculus and
  Applied Analysis 4 (2001) 47--66.

\bibitem{bungert2018solution}
L.~Bungert, M.~Burger, Solution paths of variational regularization methods for
  inverse problems, arXiv preprint arXiv:1808.01783.

\bibitem{euler1738progressionibus}
L.~Euler, De progressionibus transcendentibus seu quarum termini generales
  algebraice dari nequeunt, Commentarii academiae scientiarum Petropolitanae
  (1738) 36--57.

\end{thebibliography}

\end{document}